\documentclass[12pt]{article} %\openup 10pt
\usepackage{times,amsmath,amsbsy,amssymb,amsthm}
\usepackage{graphicx}
\usepackage{dsfont}
\usepackage{subfigure}
\newtheorem{theorem}{Theorem}[section]
\newtheorem{lemma}{Lemma}[section]
\newtheorem{remark}{Remark}[section]
\numberwithin{equation}{section}
\begin{document}

  \title{
Uniform convergence and a posteriori error estimation for assumed stress hybrid finite element
 methods}
%Natural Science Foundation of China
%   (10771150), the National Basic Research Program of China (2005CB321701),
%    and the Program for New Century Excellent Talents in University (NCET-07-0584) }}
\author{  Guozhu Yu$^1$
 \ ,\quad
 Xiaoping Xie$^{1}$\thanks{Corresponding author}
 \ ,\quad Carsten Carstensen$^{2,3}$
 \\
 { \small $_1$ \ School of Mathematics, Sichuan University, Chengdu 610064,  China}\\
{\small Email: yuguozhumail@yahoo.com.cn, xpxiec@gmail.com}\\
 { \small $_2$  \  Institut f$\ddot{\mbox{u}}$r Mathematik, Humboldt Universit$\ddot{\mbox{a}}$t zu Berlin, Unter den Linden 6,
       10099 Berlin, Germany}\\
   {\small $3$ \ Department of
       Computational Science and Engineering, Yonsei University, 120-749 Seoul, Korea }    \\
 {\small Email: cc@math.hu-berlin.de}\\
}
%\address[C.~Carstensen]{Humboldt-Universit\"at zu Berlin, Unter den Linden 6,
%       10099 Berlin, Germany; \newline
%       Department of
%       Computational Science and Engineering, Yonsei University, 120-749 Seoul, Korea. }
%       \email{cc@math.hu-berlin.de}
\date{}%\today

\maketitle \vspace*{-9mm}

 \begin{center}
 \begin{minipage}{5.0in}
%{ \small
{\bf \textbf{Abstract}:}\,\,\, Assumed stress hybrid  methods are known to
 improve the performance of standard   displacement-based finite elements
 and are widely used in    computational mechanics. The methods are based on the Hellinger-Reissner
 variational principle for the displacement
 and stress variables.  
This work analyzes two existing 4-node hybrid stress
 quadrilateral elements due to Pian and Sumihara [Int. J. Numer. Meth. Engng, 1984] and due to Xie and Zhou [Int. J. Numer. Meth. Engng, 2004], which behave robustly in numerical benchmark
 tests.    For the finite elements,  the isoparametric bilinear  interpolation is used for the  displacement approximation,  while different piecewise-independent 5-parameter modes are employed for the   stress approximation. 
 We show   that  the two schemes are free from Poisson-locking,  in the sense that the error bound in the a
 priori estimate is independent of the relevant Lam$\acute{\mbox{e}}$
 constant $\lambda$. We also   establish the equivalence  of the methods to two assumed enhanced strain schemes.  Finally, we derive reliable and efficient residual-based a
posteriori error estimators for  the stress in $L^{2}$-norm and the displacement  in $H^{1}$-norm, and verify the theoretical results by some numerical experiments.    

\par
\vspace{0.5cm}

\end{minipage}
\end{center}
 %***************************************SS1

 \setcounter{remark}{0} \setcounter{lemma}{0} \setcounter{theorem}{0}
\setcounter{section}{1} \setcounter{equation}{0}
\section*{1. \ Introduction}
Let $\Omega\subset \mathbb{R}^{2}$ be a bounded open set with
boundary $\Gamma=\Gamma_{D}\bigcup\Gamma_{N}$, where
meas($\Gamma_{D}$)$>$0. The plane linear elasticity model is given
by
\begin{equation}\label{model1}
\left\{
\begin{array}{ll}
 -\mathbf{div}\boldsymbol{\sigma}
=\mathbf{f}  &\mbox{in}\quad\Omega,\\
\boldsymbol{\sigma}=\mathbb{C}\varepsilon(\mathbf{u}) &\mbox{in}\quad\Omega,\\
 \mathbf{u}|_{\Gamma_{D}}=0,
 \boldsymbol{\sigma}\mathbf{n}|_{\Gamma_{N}}=\mathbf{g},\\
\end{array}
\right.
\end{equation}
where $\boldsymbol{\sigma}\in\mathbb{R}^{2\times2}_{sym}$ denotes the symmetric stress tensor field,
 $\mathbf{u}\in\mathbb{R}^{2}$  the displacement field,
 $\varepsilon(\mathbf{u})=(\nabla \mathbf{u}+\nabla^{T}\mathbf{u})/2$  the strain,
 $\mathbf{f}\in\mathbb{R}^{2}$  the body loading density,
 $\mathbf{g} \in\mathbb{R}^{2}$ the surface traction,
 $\mathbf{n}$  the unit outward vector normal to $\Gamma$,
 and $\mathbb{C}$ the elasticity modulus  tensor with 
$$\mathbb{C}\varepsilon(\mathbf{u})=2\mu \varepsilon(\mathbf{u})+\lambda \mbox{div}\mathbf{u}\ \mathbf{I},$$ $\mathbf{I}$  the $2\times 2$ identity tensor,
and $\mu,\lambda$  the Lam$\acute{\mbox{e}}$ parameters given by
 $\mu=\frac{E}{2(1+\nu)}$, $\lambda=\frac{E\nu}{(1+\nu)(1-2\nu)}$ for plane strain problems and by $\mu=\frac{E}{2(1+\nu)}$, $\lambda=\frac{E\nu}{(1+\nu)(1-\nu)}$ for plane stress problems,
with $0<\nu<0.5$  the Poisson ratio and $E$
the Young's modulus.

It is well-known  that the standard 4-node displacement quadrilateral element (i.e.
 isoparametric bilinear element) yields poor results at coarse meshes for   problems with bending and suffers from  "Poisson locking"  for   plane strain problems, at the nearly
incompressible limit  ($\lambda\rightarrow\infty$ as $\nu\rightarrow 0.5$).    We refer  to \cite{Babuska-Suri1992} for   the mathematical characteristic of
 locking.
To improve the performance of the isoparametric bilinear displacement
 element while preserving its convenience, 
various methods  have been suggested
 in literature.   
 
 The method of incompatible
 displacement modes is based on  enriching the standard displacement modes
 with internal incompatible displacements.  A representative incompatible displacement is  the so-called Wilson  element proposed by Wilson, Taylor, Doherty, and Ghaboussi \cite{Wilson-Taylor-Doherty-Ghaboussi1973}.  It achieves a greater degree of accuracy than the  isoparametric bilinear element when using coarse meshes.  This element was
subsequently modified by Taylor, Wilson and Beresford \cite{Taylor-Wilson-Beresford1976}, and the modified Wilson element behaves uniformly in the nearly incompressibility.    In  \cite{Lesaint1976},  Lesaint
 analyzed convergence on uniform square meshes for Wilson element.  He and Zl$\acute{\mbox{a}}$mal then established convergence for the modified
 Wilson element on arbitrary quadrilateral meshes \cite{Lesaint-Zlamal1980}.
 In \cite{Shi1984}, Shi established a convergence condition for the quadrilateral Wilson
 element.  In  \cite{Zhang1997}, Zhang derived uniform convergence for the modified
 Wilson element on arbitrary quadrilateral meshes.
 
 The assumed-stress hybrid approach is a kind of mixed method based on the Hellinger-Reissner variational principle which
 includes displacements and stresses.   The  pioneering work in this direction is by Pian \cite{Pian1964}, where the assumed
stress field assumed to satisfy the homogenous equilibrium equations pointwise.  In \cite{Pian-Chen1982} Pian and Chen proposed
a new type of the hybrid-method by imposing the stress equilibrium equations in a variational
sense and by adopting the natural co-ordinate for stress approximation.  In  \cite{Pian-Sumihara1984} Pian and Sumihara  derived the famous assumed
stress hybrid element (abbreviated as the PS finite element) through a rational choice of stress terms.  Despite of the use of isoparametric
bilinear displacement approximation,  the  PS finite element yields uniformly accurate results for all the numerical benchmark  tests.  Pian and Tong \cite{Pian-Tong1989} discussed the similarity and  basic difference between the incompatible
displacement model and the hybrid stress model. In the direction  of determining the optimal stress parameters, there have been many other research efforts \cite{Pian-Wu1988,Xie-Zhou2004, Xie2005,Xie-Zhou2008,Zhou-Nie2001}.  In \cite{Xie-Zhou2004, Xie-Zhou2008}, Xie and Zhou derived robust 4-node hybrid stress quadrilateral elements by optimizing stress modes with a so-called energy-compatibility condition,  i.e. the assumed stress terms are orthogonal to the enhanced strains caused by Wilson bubble displacements.  
In \cite{Zhou-Xie2002}  a  convergence  analysis was established for the PS  element, but the upper bound in the error estimate  is not uniform with respect to $\lambda$.  So far there is no  uniform error analysis with respect to the nearly incompressibility for the assumed stress hybrid methods on arbitrary quadrilateral meshes.
 
 Closely related to the assumed stress method is the enhanced assumed strain method (EAS) pioneered by Simo
and Rifai \cite{Simo-Rifai1990}. Its variational basis is  the Hu-Washizu principle which includes displacements, stresses, and
 enhanced strains. It was shown in  \cite{Simo-Rifai1990} that the classical method of
incompatible displacement modes is a special case of the EAS-method. Yeo and Lee  \cite{Yeo-Lee1996} proved that the EAS concept in some model 
 situation is equivalent to a Hellinger-Reissner formulation. In \cite{Reddy-Simo1995}, Reddy and Simo established an a priori error estimate for the EAS method on
 parallelogram meshes.   Braess \cite{Braess1998} re-examined
 the sufficient conditions for convergence, in particular relating the stability
 condition to a strengthened Cauchy inequality, and elucidating the
 influence of the Lam$\acute{e}$ constant $\lambda$. In \cite{Braess-Carstensen-Reddy2004}, Braess, Carstensen and Reddy
 established uniform convergence and a posteriori estimates for the EAS method on
 parallelogram meshes.

 The main goal of this work is to establish
 uniform convergence and a posteriori error estimates for two 4-node assumed stress hybrid
 quadrilateral elements: the  PS finite element by Pian and Sumihara \cite{Pian-Sumihara1984}
 and the ECQ4 finite element by Xie and Zhou  \cite{Xie-Zhou2004}.  Equivalence is established between the hybrid finite element schemes and  two EAS proposed schemes. We also carry out an a posteriori error analysis for the hybrid methods.

The paper is organized as follows. In Section 2 we  discuss the uniform stability of the weak
 formulations. Section 3 is devoted to  finite element formulations of the hybrid elements PS and ECQ4 and their numerical   performance investigation. We establish the uniformly stability conditions and derive uniform a priori error estimates in Section 4.   Equivalence between the hybrid schemes and  two EAS schemes is discussed in Section 5. We devote Section 6 to an analysis of a posteriori error estimates for the hybrid methods and verification of the theoretical results  
 by numerical tests.

 \setcounter{remark}{0}\setcounter{lemma}{0} \setcounter{theorem}{0}
\setcounter{section}{2} \setcounter{equation}{0}
\section*{2. \  Uniform stability of the weak formulations}

First we introduce some notations. Let $L^{2}(T;X)$ be  the space of square
integrable functions defined on $T$ with values in the finite-dimensional vector space X and with norm being denoted by $||\cdot||_{0,T}$. We denote by
$H^{k}(T;X)$ the usual Sobolev space consisting of functions defined on $T$, taking values
in $X$, and with all derivatives of order up to  $k$ square-integrable. The  norm  on $H^{k}(T;X)$ is denoted by $||\cdot||_{k,T}:= (\sum_{0\leq j\leq k}|v|_{j,T}^{2})^{1/2}, $ with $|\cdot|_{k,T} $ the  semi-norm  derived from the partial derivatives of order equal to $k$.  
% anddefined respectively as
%$$|v|_{k,K}^{2} :=\int_{K}\sum_{|\alpha|=k}|D^{\alpha}v|^{2}dx,\ \ ||v||_{k,K}^{2}:=$$
When there is no conflict, we may abbreviate them to $||\cdot||_{k}$ and $|\cdot|_{k}$.   Let $L^{2}_{0}(\Omega)$ be the space of  square
integrable functions with zero
mean values. We denote by $P_{k}(T)$ the set of polynomials of degree less than or equal to $k$, by $Q_{k}$ the set of polynomials of degree less than or equal to $k$ in each variable. 

For convenience, we use the notation $a\lesssim b$  to represent that there
 exists a generic positive constant $C$, independent of the mesh parameter
 $h$ and the Lam$\acute{e}$ constant $\lambda$, such that
 $a \leq Cb$. Finally, $a\approx b$ abbreviates $a\lesssim b\lesssim a$.

We define  two spaces as follows:
 $$V:=H_{D}^{1}(\Omega)^{2}=\{\mathbf{u}\in H^{1}(\Omega)^{2}: \mathbf{u}|_{\Gamma_{D}}=0\},$$
 \[
 \Sigma:=\left\{
\begin{array}{ll}
 \mathbf{L}^{2}(\Omega;\mathbb{R}_{sym}^{2\times2}), &if\,\, meas(\Gamma_{N})>0,\\
 \{\boldsymbol{\tau} \in \mathbf{L}^{2}(\Omega;\mathbb{R}_{sym}^{2\times2}): \int_{\Omega}tr\boldsymbol{\tau} d\mathbf{x}=0\}, &if\,\,
 \Gamma_{N}=\emptyset,
\end{array}
\right.
\]
where $\mathbf{L}^{2}(\Omega;\mathbb{R}_{sym}^{2\times2})$
 denotes the space of square-integrable symmetric tensors with the norm $||\cdot||_{0}$ defined by $||\boldsymbol{\tau}||_{0}^{2}:=\int_{\Omega}\boldsymbol{\tau}:\boldsymbol{\tau} d{\bf x}$, and $tr \boldsymbol{\tau}:=\boldsymbol{\tau}_{11}+\boldsymbol{\tau}_{22}$ represents the trace of the tensor $\boldsymbol{\tau}$.  Notice that on the space $V$, the semi-norm $|\cdot|_{1}$ is equivalent to the norm $||\cdot||_{1}$.
 
 The Hellinger-Reissner variational principle for the model
 (\ref{model1}) reads as:  Find $(\boldsymbol{\sigma},\mathbf{u})\in \Sigma\times V$ with
 \begin{equation}\label{weak1.a}
 a(\boldsymbol{\sigma},\boldsymbol{\tau})-\int_{\Omega}\boldsymbol{\tau}:\varepsilon(\mathbf{u})d\mathbf{x}=0 \hskip5mm\,\mbox{for all }\boldsymbol{\tau}\in
\Sigma,
\end{equation}
\begin{equation}\label{weak1.b}
 \int_{\Omega}\boldsymbol{\sigma}:\varepsilon(\mathbf{v})d\mathbf{x}
 =F(\mathbf{v})\,  \hskip1.3cm
 \mbox{for all } \mathbf{v}\in V,
 \end{equation}
where
\begin{eqnarray*}
a(\boldsymbol{\sigma},\boldsymbol{\tau}):&=&\int_{\Omega}\boldsymbol{\sigma}:\mathbb{C}^{-1}\boldsymbol{\tau} d\mathbf{x}=\frac{1}{2\mu}\int_{\Omega}\left(\boldsymbol{\sigma}:\boldsymbol{\tau}-\frac{\lambda}{2(\mu+\lambda)}tr\boldsymbol{\sigma} tr\boldsymbol{\tau}\right) d\mathbf{x}\\
&=&\int_{\Omega}\left(\frac{1}{2\mu}\boldsymbol{\sigma}^{D}:\boldsymbol{\tau}^{D}+\frac{1}{4(\mu+\lambda)}tr\boldsymbol{\sigma} tr\boldsymbol{\tau}\right) d\mathbf{x},\\
F(\mathbf{v}):&=&\int_{\Omega}\mathbf{f}\cdot\mathbf{v}d\mathbf{x}+\int_{\Gamma_{N}}\mathbf{g}\cdot\mathbf{v}ds.\end{eqnarray*}
Here and throughout the paper,
$\boldsymbol{\sigma}:\boldsymbol{\tau}=\sum_{i,j=1}^{2}\boldsymbol{\sigma}_{ij}\boldsymbol{\tau}_{ij},$
 and   $\boldsymbol{\tau}^{D}:=\boldsymbol{\tau}-\frac{1}{2}tr\boldsymbol{\tau} \mathbf{I}$.

 The following continuity conditions are immediate:
 \begin{equation}\label{a-continuity} 
 a(\boldsymbol{\sigma},\boldsymbol{\tau})\lesssim ||\boldsymbol{\sigma}||_{0}||\boldsymbol{\tau}||_{0},\ \boldsymbol{\sigma},\ \boldsymbol{\tau}\in \Sigma,
 \end{equation}
 \begin{equation}\label{b-continuity} 
 \int_{\Omega}\boldsymbol{\tau}:\varepsilon(\mathbf{v})d\mathbf{x}\lesssim ||\boldsymbol{\tau}||_{0}|\mathbf{v}|_{1},\  \boldsymbol{\tau}\in \Sigma,\ \mathbf{v}\in V,
 \end{equation} 
 \begin{equation}\label{F-continuity} 
 F(\mathbf{v})\lesssim (||\mathbf{f}||_{-1}+||\mathbf{g}||_{-\frac{1}{2},\Gamma_{N}})|\mathbf{v}|_{1},\   \mathbf{v}\in V.
 \end{equation}

 According to the theory of mixed finite element methods \cite{Brezzi1974,Brezzi-Fortin1991}, we need the following two  stability conditions for the  
 well-posedness of the weak problem
 (\ref{weak1.a})-(\ref{weak1.b}).\\
  ($\mathrm{A1}$) \ Kernel-coercivity: For any $\boldsymbol{\tau}\in Z:=\{\boldsymbol{\tau}\in \Sigma:
 \int_{\Omega}\boldsymbol{\tau}:\varepsilon(\mathbf{v})d\mathbf{x}=0\ \ \mbox{for all } \mathbf{v}\in V\}$
 it holds
 $$\|\boldsymbol{\tau}\|_{0}^{2}\lesssim a(\boldsymbol{\tau},\boldsymbol{\tau});$$
 ($\mathrm{A2}$) \ Inf-sup condition: For any $\mathbf{v}\in V$ it holds
 $$ |\mathbf{v}|_{1}\lesssim \sup_{0\neq\boldsymbol{\tau}\in \Sigma}
 \frac{\int_{\Omega}\boldsymbol{\tau}:\varepsilon(\mathbf{v})d\mathbf{x}}{\|\boldsymbol{\tau}\|_{0}} .$$
 
 The proof of  (A1)-(A2) utilizes a lemma of Bramble, Lazarov and Pasciak.

\begin{lemma} (\cite{Bramble-Lazarov-Pasciak2001}) \ 
 For
 $q\in L:=
 \left\{
\begin{array}{ll}
 L^{2}(\Omega)\ \ &if\,\, meas(\Gamma_{N})>0,\\
 L_{0}^{2}(\Omega)\ \  &if\,\, \Gamma_{N}=\emptyset
\end{array}
\right.$ it holds
\begin{eqnarray*}
  \|q\|_{0}\lesssim\sup\limits_{\mathbf{v}\in V}
 \frac{\int_{\Omega}q \, {\bf div}\mathbf{v}d\mathbf{x}}{|\mathbf{v}|_{1}}
.
\end{eqnarray*}
% where $L_{0}^{2}(\Omega):=\{q\in L^{2}(\Omega): \int_{\Omega}qd\mathbf{x}=0\}.$
\end{lemma}

The following stability result is given in \cite{Braess-Carstensen-Reddy2004} for the model situation $\Gamma_{N}=\emptyset$.
\begin{theorem}
The uniform stability conditions (A1) and (A2) hold.
\end{theorem}
\begin{proof}
 Firstly we prove (A1). 
 Since
 $$a(\boldsymbol{\tau},\boldsymbol{\tau})=\int_{\Omega}\left(\frac{1}{2\mu}\boldsymbol{\tau}^{D}:\boldsymbol{\tau}^{D}+\frac{1}{4(\mu+\lambda)}tr\boldsymbol{\tau} tr\boldsymbol{\tau}\right) d\mathbf{x},$$
 we only need to prove $\|tr\boldsymbol{\tau}\|_{0}\lesssim\|\boldsymbol{\tau}^{D}\|_{0}$ for any $\boldsymbol{\tau}\in Z$.
  
  In fact, for  $\boldsymbol{\tau}\in Z$ and any $\mathbf{v}\in V$, it holds
\begin{eqnarray*}
 0 &= &\int_{\Omega}\boldsymbol{\tau}:\varepsilon(\mathbf{v})d\mathbf{x}=\int_{\Omega}(\frac{1}{2}tr\boldsymbol{\tau}\mathbf{I}+\boldsymbol{\tau}^{D}):\varepsilon(\mathbf{v})d\mathbf{x}\\
   &= &\int_{\Omega}\frac{1}{2}tr\boldsymbol{\tau}
        \mbox{div}\mathbf{v}d\mathbf{x}+\int_{\Omega}\boldsymbol{\tau}^{D}:\varepsilon(\mathbf{v})d\mathbf{x}.
\end{eqnarray*}
 Thus,  by Lemma 2.1 we obtain
\begin{eqnarray*}
 \|tr\boldsymbol{\tau}\|_{0}\lesssim \sup\limits_{\mathbf{v}\in V}
 \frac{\int_{\Omega}tr\boldsymbol{\tau}\ \mbox{div}\mathbf{v}d\mathbf{x}}{|\mathbf{v}|_{1}}
 =\sup\limits_{\mathbf{v}\in V}
 \frac{-2\int_{\Omega}\boldsymbol{\tau}^{D}:\varepsilon(\mathbf{v})d\mathbf{x}}{|\mathbf{v}|_{1}}
  \leq 2\|\boldsymbol{\tau}^{D}\|_{0}.
\end{eqnarray*}
This implies (A1).  For the proof of (A2), let $\mathbf{v} \in V$ and notice $\varepsilon(\mathbf{v})\in \Sigma$. Then
 $$
|\varepsilon(\mathbf{v})|_{0}\leq \sup_{\boldsymbol{\tau}\in \Sigma\setminus\{0\}}
 \frac{\int_{\Omega}\boldsymbol{\tau}:\varepsilon(\mathbf{v})d\mathbf{x}}{\|\boldsymbol{\tau}\|_{0}}
.$$
Hence (A2) follows from   the equivalence between the two norms $|\varepsilon(\mathbf{v})|_{0}$ and $ |\mathbf{v}|_{1}$ on $V$.
\end{proof}

In view of the continuity conditions, (\ref{a-continuity})-(\ref{F-continuity}), and the stability conditions, (A1)-(A2), we immediately get the well-posedness results:
\begin{theorem}
Assume that $\mathbf{f}\in V',\ \mathbf{g}\in H^{-1/2}(\Gamma_{N})$. Then the   weak problem  (\ref{weak1.a})-(\ref{weak1.b}) admits a unique solution $(\boldsymbol{\sigma},\mathbf{u})\in \Sigma\times V$ such that
$$||\boldsymbol{\sigma}||_{0}+|\mathbf{u}|_{1}\lesssim ||\mathbf{f}||_{-1}+||\mathbf{g}||_{-\frac{1}{2},\Gamma_{N}}. $$
\end{theorem}

\setcounter{remark}{0}\setcounter{lemma}{0} \setcounter{theorem}{0}
\setcounter{section}{3} \setcounter{equation}{0}
\section*{3. \ Finite element formulations for hybrid methods}

\subsection *{3.1 \ Geometric properties of quadrilaterals}
In what follows we assume that $\Omega$ is a convex polygonal domain. Let $T_{h}$ be a conventional quadrilateral mesh  of ${\Omega}$.
We denote by $h_{K}$  the diameter of a quadrilateral $K\in T_{h}$, and denote  $h:=\max_{K\in T_{h}}h_{K}$.  Let $Z_{i}(x_{i},y_{i})$, $1\leq i\leq 4$ be
 the four vertices of $K$, and  $T_{i}$ denotes the sub-triangle
 of $K$ with vertices $Z_{i-1}$, $Z_{i}$ and $Z_{i+1}$ 
 (the index on $Z_{i}$ is modulo 4). Define
\begin{equation*}
 \rho_{K}=\min\limits_{1\leq i\leq4}\ \mathrm{diameter\ of\ circle\ inscribed\
 in}\ T_{i}.
\end{equation*}
Throughout the paper, we assume that the partition $T_{h}$ satisfies the
following "shape-regularity" hypothesis: There exist a constant $\varrho>2$ independent
of $h$ such that for all $K\in T_{h},$
\begin{equation}\label{partition condition}
 h_{K}\leq \varrho \rho_{K}.
\end{equation} 

\begin{remark}
As pointed out in \cite{Zhang1997}, this shape regularity condition is equivalent to the following one which has been widely used
in  literature (e.g. \cite{Ciarlet1978}): there exist two constants $\varrho'>2$ and $0<\gamma<1$ independent of $h$
such that for all $K\in T_{h}$,
\begin{equation*}%\label{CR}
 h_{K}\leq \varrho'\rho_{K}',\ \ | \cos\theta_{K}^{i}|\leq \gamma \  \ \ \mbox{ for } \ \ 1\leq i\leq 4.
 \end{equation*}
Here $\rho_{K}'$ and $\theta_{K}^{i}$
denote the maximum diameter of all circles contained in $K$ and the
angles associated with vertices of $K$.
\end{remark}

 Let $\hat{K}=[-1,1]\times[-1,1]$ be the reference square with
 vertices $\hat{Z}_{i}$, $1\leq i\leq 4$. Then exists a unique
 invertible mapping $F_{K}$ that maps $\hat{K}$ onto
 $K$ with $F_{K}(\xi,\eta)\in Q_{1}^{2}(\xi,\eta)$ and $F_{K}(\hat{Z}_{i})=Z_{i}$, $1\leq i\leq 4$ (Figure 1). Here $\xi, \eta\in [-1,1]$ are the local isoparametric coordinates.

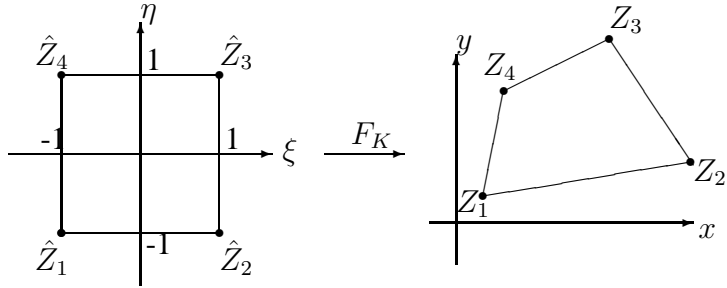
\begin{figure}[!h]
\begin{center}
\setlength{\unitlength}{0.7cm}
\begin{picture}(10,6)
\put(0,1.5){\line(1,0){3}}        \put(3,1.5){\line(0,1){3}}
\put(0,1.5){\line(0,1){3}} \put(0,4.5){\line(1,0){3}}
\put(0,1.5){\circle*{0.15}}      \put(0,4.5){\circle*{0.15}}
\put(3,4.5){\circle*{0.15}}       \put(0,1.5){\circle*{0.15}}
\put(3,1.5){\circle*{0.15}} \put(-0.5,0.8){\bf{$\hat{Z}_{1}$}}
\put(3,0.8){\bf{$\hat{Z}_{2}$}}    \put(3,4.7){\bf{$\hat{Z}_{3}$}}
\put(-0.5,4.7){\bf{$\hat{Z}_{4}$}}  \put(-1,3){\vector(1,0){5}}
\put(1.5,0.5){\vector(0,1){5}} \put(4.2,2.9){$\xi$}
\put(1.5,5.6){$\eta$}

\put(1.6,1.1){-1}                 \put(1.6,4.6){1}
\put(-0.4,3.1){-1} \put(3.1,3.1){1}

\put(5,3){\vector(1,0){1.5}}      \put(5.5,3.2){$F_K$}

\put(8,2.2){\line(6,1){4}}        \put(8,2.2){\line(1,5){0.4}}
\put(8.4,4.2){\line(2,1){2}} \put(10.4,5.2){\line(2,-3){1.6}}
\put(8,2.2){\circle*{0.15}}      \put(8.4,4.2){\circle*{0.15}}
\put(10.4,5.2){\circle*{0.15}}    \put(11.95,2.85){\circle*{0.15}}
\put(7.5,1.9){\bf{$Z_{1}$}} \put(12,2.5){\bf{$Z_{2}$}}
\put(10.4,5.4){\bf{$Z_{3}$}}       \put(8.0,4.5){\bf{$Z_{4}$}}
\put(7,1.7){\vector(1,0){5}}      \put(7.5,0.9){\vector(0,1){4}}
\put(12.1,1.4){$x$} \put(7.5,5.0){$y$}

\end{picture}
\end{center}
\vspace{-1cm} \caption{The mapping $F_{K}$}
\end{figure}

 This  isoparametric bilinear mapping $(x,y)=F_{K}(\xi,\eta)$  is given by
\begin{equation}\label{relation1}
 x=\sum_{i=1}^{4}x_{i}N_{i}(\xi,\eta),\,\,\,
 y=\sum_{i=1}^{4}y_{i}N_{i}(\xi,\eta),
\end{equation}
 where
 $$N_{1}=\frac{1}{4}(1-\xi)(1-\eta),\,
   N_{2}=\frac{1}{4}(1+\xi)(1-\eta),\,
   N_{3}=\frac{1}{4}(1+\xi)(1+\eta),\,
   N_{4}=\frac{1}{4}(1-\xi)(1+\eta).$$
 We can rewrite (\ref{relation1}) as
\begin{equation}\label{relation2}
 x=a_{0}+a_{1}\xi+a_{2}\eta+a_{12}\xi\eta,\,\,\,
 y=b_{0}+b_{1}\xi+b_{2}\eta+b_{12}\xi\eta,
\end{equation}
where
\begin{equation*}
\left(\begin{array}{cc}
 a_{0} &b_{0}\\
 a_{1} &b_{1}\\
 a_{2} &b_{2}\\
 a_{12} &b_{12}
\end{array}\right)=\frac{1}{4}
\left(\begin{array}{cccc}
  1 &1  &1 &1\\
 -1 &1  &1 &-1\\
 -1 &-1 &1 &1\\
  1 &-1 &1 &-1\\
\end{array}\right)
\left(\begin{array}{cc}
 x_{1} &y_{1}\\
 x_{2} &y_{2}\\
 x_{3} &y_{3}\\
 x_{4} &y_{4}\\
\end{array}\right).
\end{equation*}
\begin{remark} Due to the choice of  node order (Figure 1), we always have 
$a_{1}>0, b_{2}>0.$
\begin{remark} Notice that when $K$ is a parallelogram,  we have $a_{12}=b_{12}=0$, and  $F_{K}$ is reduced to an affine mapping.
 \end{remark}

\end{remark}
 Then the Jacobi matrix of the transformation $F_{K}$ is
 \begin{equation*}
 DF_{K}(\xi,\eta)=
 \left(\begin{array}{cc}
  \frac{\partial x}{\partial \xi} &\frac{\partial x}{\partial \eta}\\
  \frac{\partial y}{\partial \xi} &\frac{\partial y}{\partial \eta}\\
 \end{array}\right)=
 \left(\begin{array}{cc}
  a_{1}+a_{12}\eta &a_{2}+a_{12}\xi\\
  b_{1}+b_{12}\eta &b_{2}+b_{12}\xi\\
 \end{array}\right),
 \end{equation*}
 and the Jacobian of $F_{K}$ is
 \begin{equation*}
  J_{K}(\xi,\eta)=det(DF_{K})=J_{0}+J_{1}\xi+J_{2}\eta,
 \end{equation*}
 where
 \begin{equation*}
  J_{0}=a_{1}b_{2}-a_{2}b_{1},\,J_{1}=a_{1}b_{12}-a_{12}b_{1},\,J_{2}=a_{12}b_{2}-a_{2}b_{12}.
 \end{equation*}
 Denote by $F_{K}^{-1}$ the inverse of $F_{K}$, then we obtain
 \begin{eqnarray*}
 \left(\begin{array}{cc}
  \frac{\partial \xi}{\partial x} &\frac{\partial \xi}{\partial y}\\
  \frac{\partial \eta}{\partial x} &\frac{\partial \eta}{\partial y}\\
 \end{array}\right)
 &= &DF_{K}^{-1}\circ F_{K}(\xi,\eta)=(DF_{K})^{-1}\\
 &= &\frac{1}{J_{K}(\xi,\eta)}
 \left(\begin{array}{cc}
  b_{2}+b_{12}\xi &-a_{2}-a_{12}\xi\\
  -b_{1}-b_{12}\eta &a_{1}+a_{12}\eta\\
 \end{array}\right).
 \end{eqnarray*}

It holds the following element geometric properties:

 \begin{lemma} \label{zhang} (\cite{Zhang1997}) \  For any $K\in T_{h}$, under the
 hypothesis (\ref{partition condition}), we have
 \begin{equation}\label{partition satisfy1}
 \frac{\max\limits_{(\xi,\eta)\in \hat{K}} J_{K}(\xi,\eta)}{\min\limits_{(\xi,\eta)\in \hat{K}} J_{K}(\xi,\eta)}
 <\frac{h_{K}^{2}}{2\rho_{K}^{2}}\leq \frac{\varrho^{2}}{2},
 \end{equation}
   \begin{equation}\label{a1b1-a2b2}
 \frac{1}{4}\rho_{K}^{2}<a_{1}^{2}+b_{1}^{2}<\frac{1}{4}h_{K}^{2},
\ \ 
 \frac{1}{4}\rho_{K}^{2}<a_{2}^{2}+b_{2}^{2}<\frac{1}{4}h_{K}^{2}, \ \
% \end{equation}
%  \begin{equation}\label{partition satisfy2}
 a_{12}^{2}+b_{12}^{2}<\frac{1}{16}h_{K}^{2}.
 \end{equation}
\end{lemma}

  In view of the choice of  node order (cf. Figure 1), the shape-regular hypothesis (\ref{partition condition}) and the relations (\ref{a1b1-a2b2}), without loss of generality we assume 
 \begin{equation}\label{aibi-assumption}
 |b_{1}|\leq a_{1},\ \ |a_{2}|\lesssim b_{2}. 
\end{equation}
Together with (\ref{a1b1-a2b2}), this leads to
\begin{equation}\label{a1b2-hK}
a_{1}\approx b_{2}\approx h_{K},\ \ \max\{a_{2},b_{1}\}\lesssim O( h_{K}).
\end{equation}
Notice also that  Lemma \ref{zhang}   shows
\begin{equation}\label{JK-hK}
J_{K}\approx J_0\approx h_{K}^{2}.
\end{equation}

 \subsection *{3.2 \ Hybrid methods PS   and ECQ4}

This subsection is devoted to the finite element formulations of the 4-node assumed stress hybrid quadrilateral elements PS \cite{Pian-Sumihara1984} and ECQ4 \cite{Xie-Zhou2004}.

 Let $\Sigma_{h}\subset\Sigma$ and $V_{h}\subset V$ be finite
 dimensional spaces respectively for stress and displacement approximations,
 then the corresponding finite element scheme for the problem
 (\ref{weak1.a})(\ref{weak1.b}) reads as:
 Find $(\boldsymbol{\sigma}_{h},\mathbf{u}_{h})\in \Sigma_{h}\times V_{h}$, such that
 \begin{equation}\label{discreteweak1.a}
 a(\boldsymbol{\sigma}_{h},\boldsymbol{\tau})-\int_{\Omega}\boldsymbol{\tau}:\varepsilon(\mathbf{u}_{h})d\mathbf{x}=0 \ \ \,\mbox{for all }\boldsymbol{\tau}\in
\Sigma_{h},
\end{equation}
\begin{equation}\label{discreteweak1.b}
 \int_{\Omega}\boldsymbol{\sigma}_{h}:\varepsilon(\mathbf{v})d\mathbf{x}
 =F(\mathbf{v})\,\hskip1cm
\ \  \mbox{for all } \mathbf{v}\in V_{h}.
 \end{equation}

 For elements PS and ECQ4, the isoparametric bilinear interpolation is used for the displacement approximation, i.e. the displacement space $V_{h}$ is   chosen as
 \begin{equation}\label{displacement}
 V_{h}=\{\mathbf{v}\in V: \hat{\mathbf{v}}=\mathbf{v}|_{K}\circ
 F_{K}\in Q_{1}(\hat{K})^{2}\ \ \mbox{for all } K \in T_{h}\}.
 \end{equation}
 In other words,  for $\mathbf{v}=(u,v)^{T}\in V_{h}$ with nodal values $\mathbf{v}(Z_{i})=(u_{i},v_{i})^{T}$ on $K$,   
\begin{equation}
\label{displacementform}
\hat{\mathbf{v}}=\sum_{i=1}^{4}\left(\begin{array}{l}u_{i}\\v_{i}
\end{array}\right)N_{i}(\xi,\eta)=\left(\begin{array}{l}U_{0}+U_{1}\xi+U_{2}\eta+U_{12}\xi\eta\\V_{0}+V_{1}\xi+V_{2}\eta+V_{12}\xi\eta
\end{array}\right),
 \end{equation}
  where
 \begin{equation*}
\left(\begin{array}{cc}
 U_{0} &V_{0}\\
 U_{1} &V_{1}\\
 U_{2} &V_{2}\\
 U_{12}&V_{12}
\end{array}\right)=\frac{1}{4}
\left(\begin{array}{cccc}
  1 &1  &1 &1\\
 -1 &1  &1 &-1\\
 -1 &-1 &1 &1\\
  1 &-1 &1 &-1\\
\end{array}\right)
\left(\begin{array}{cc}
 u_{1} &v_{1}\\
 u_{2} &v_{2}\\
 u_{3} &v_{3}\\
 u_{4} &v_{4}\\
\end{array}\right).
\end{equation*}

 We denote the symmetric stress tensor $\boldsymbol{\tau}:=
 \left(
 \begin{array}{cc}
  \boldsymbol{\tau}_{11} &\boldsymbol{\tau}_{12}\\
  \boldsymbol{\tau}_{12} &\boldsymbol{\tau}_{22}\\
 \end{array}
 \right)$.  For convenience we abbreviate it to $\boldsymbol{\tau}=(\boldsymbol{\tau}_{11},\boldsymbol{\tau}_{22},\boldsymbol{\tau}_{12})^{T}.$ In  \cite{Pian-Sumihara1984}, the  5-parameters stress mode on $\hat{K}$ for the PS finite element takes the form
\begin{equation}\label{PSstress1}
\hat{\boldsymbol{\tau}}=
\left(\begin{array}{c}
\hat{\boldsymbol{\tau}}_{11}\\
\hat{\boldsymbol{\tau}}_{22}\\
\hat{\boldsymbol{\tau}}_{12}\\
\end{array}\right) =
\left(\begin{array}{ccccc}
1 &0 &0 &\eta &\frac{a_{2}^{2}}{b_{2}^{2}}\xi\\
0 &1 &0 &\frac{b_{1}^{2}}{a_{1}^{2}}\eta &\xi\\
0 &0 &1 &\frac{b_{1}}{a_{1}}\eta &\frac{a_{2}}{b_{2}}\xi\\
\end{array}\right)
 \beta^{\tau}\ \
 \mbox{for }\beta^{\tau}:=(\beta_{1}^\tau,\cdots,\beta_{5}^\tau)^{T}\in \mathbb{R}^{5}.
\end{equation}
 Then the corresponding stress space for the PS finite element is
 $$ \Sigma_{h}^{PS}:=\left\{\boldsymbol{\tau}\in \Sigma: \ \hat{\boldsymbol{\tau}}=\boldsymbol{\tau}|_{K}\circ F_{K}\mbox{ is of form }(\ref{PSstress1})\ \ \mbox{for all } K\in T_{h}\right\}.$$
In  \cite{Xie-Zhou2004},  the  5-parameters stress mode on $\hat{K}$ for element  ECQ4 has the form
\begin{equation}\label{ECQ4stress1}
\hat{\boldsymbol{\tau}}=
\left(\begin{array}{c}
\hat{\boldsymbol{\tau}}_{11}\\
\hat{\boldsymbol{\tau}}_{22}\\
\hat{\boldsymbol{\tau}}_{12}\\
\end{array}\right) =
\left(\begin{array}{ccccc}
1-\frac{b_{12}}{b_{2}}\xi &\frac{a_{12}a_{2}}{b_{2}^{2}}\xi &\frac{a_{12}b_{2}-a_{2}b_{12}}{b_{2}^{2}}\xi &\eta &\frac{a_{2}^{2}}{b_{2}^{2}}\xi\\
\frac{b_{1}b_{12}}{a_{1}^{2}}\eta &1-\frac{a_{12}}{a_{1}}\eta &\frac{a_{1}b_{12}-a_{12}b_{1}}{a_{1}^{2}}\eta &\frac{b_{1}^{2}}{a_{1}^{2}}\eta &\xi\\
\frac{b_{12}}{a_{1}}\eta &\frac{a_{12}}{b_{2}}\xi &1-\frac{b_{12}}{b_{2}}\xi-\frac{a_{12}}{a_{1}}\eta &\frac{b_{1}}{a_{1}}\eta &\frac{a_{2}}{b_{2}}\xi\\
\end{array}\right)
 \beta^{\tau}\ \
\mbox{for }\beta^{\tau}\in \mathbb{R}^{5}.
\end{equation}
  Then the corresponding stress space for the ECQ4 finite element is
 $$ \Sigma_{h}^{EC}:=\left\{\boldsymbol{\tau}\in \Sigma: \ \hat{\boldsymbol{\tau}}=\boldsymbol{\tau}|_{K}\circ F_{K}\mbox{ is of form }(\ref{ECQ4stress1})\ \mbox{ for all } K\in T_{h}\right\}.$$

\begin{remark}\label{stress-osc} The stress mode of ECQ4 can be viewed as a modified version of PS mode with a
 perturbation term:
 \begin{eqnarray*}
 &&\left(\begin{array}{ccccc}
  1-\frac{b_{12}}{b_{2}}\xi &\frac{a_{12}a_{2}}{b_{2}^{2}}\xi &\frac{a_{12}b_{2}-a_{2}b_{12}}{b_{2}^{2}}\xi &\eta &\frac{a_{2}^{2}}{b_{2}^{2}}\xi\\
  \frac{b_{1}b_{12}}{a_{1}^{2}}\eta &1-\frac{a_{12}}{a_{1}}\eta &\frac{a_{1}b_{12}-a_{12}b_{1}}{a_{1}^{2}}\eta &\frac{b_{1}^{2}}{a_{1}^{2}}\eta &\xi\\
  \frac{b_{12}}{a_{1}}\eta &\frac{a_{12}}{b_{2}}\xi &1-\frac{b_{12}}{b_{2}}\xi-\frac{a_{12}}{a_{1}}\eta &\frac{b_{1}}{a_{1}}\eta &\frac{a_{2}}{b_{2}}\xi\\
 \end{array}\right)\\
 &=&
 \left(\begin{array}{ccccc}
  1 &0 &0 &\eta &\frac{a_{2}^{2}}{b_{2}^{2}}\xi\\
  0 &1 &0 &\frac{b_{1}^{2}}{a_{1}^{2}}\eta &\xi\\
  0 &0 &1 &\frac{b_{1}}{a_{1}}\eta &\frac{a_{2}}{b_{2}}\xi\\
 \end{array}\right)
 +
 \left(\begin{array}{ccccc}
  -\frac{b_{12}}{b_{2}}\xi &\frac{a_{12}a_{2}}{b_{2}^{2}}\xi &\frac{a_{12}b_{2}-a_{2}b_{12}}{b_{2}^{2}}\xi &0 &0\\
  \frac{b_{1}b_{12}}{a_{1}^{2}}\eta &-\frac{a_{12}}{a_{1}}\eta &\frac{a_{1}b_{12}-a_{12}b_{1}}{a_{1}^{2}}\eta &0 &0\\
  \frac{b_{12}}{a_{1}}\eta &\frac{a_{12}}{b_{2}}\xi &-\frac{b_{12}}{b_{2}}\xi-\frac{a_{12}}{a_{1}}\eta &0 &0\\
 \end{array}\right).
 \end{eqnarray*}
\end{remark}
\begin{remark} When $K\in T_{h}$ is a parallelogram, the stress mode of ECQ4 is reduced to that of PS due to $a_{12}=b_{12}=0$. Thus,  PS and ECQ4 are equivalent on parallelogram meshes.
\end{remark}

 Define the bubble function space
 \begin{equation}\label{bubbledisplacement}
  B_{h}:=\{\mathbf{v}^{b}\in (L^{2}(\Omega))^{2}:
          \hat{\mathbf{v}}^{b}(\xi,\eta)=\mathbf{v}^{b}|_{K}\circ F_{K}\in span\{\xi^{2}-1,\eta^{2}-1\}^{2}\ \ \mbox{for all } K\in
          T_{h}\}.
 \end{equation}
Then for any $\mathbf{v}^{b}\in B_{h}$,  we have
\begin{equation}
\label{vB-form}
 \hat{\mathbf{v}}^{b}=\mathbf{v}^{b}\circ
 F_{K}
% =\frac{\mathbf{v}_{\xi}}{2}(\xi^{2}-1)+\frac{\mathbf{v}_{\eta}}{2}(\eta^{2}-1)
 =\left(\begin{array}{c}
\frac{u_{\xi}}{2}(\xi^{2}-1)+\frac{u_{\eta}}{2}(\eta^{2}-1)\\
\frac{v_{\xi}}{2}(\xi^{2}-1)+\frac{v_{\eta}}{2}(\eta^{2}-1)\\
\end{array}\right)
 \end{equation}
 with $u_{\xi},u_{\eta},v_{\xi},v_{\eta}\in R$.
 \begin{remark} It is easy to know (see \cite{Shi1984}) that 
            for any $K\in T_{h}$,
             $|u_{\xi}|+|u_{\eta}|+|v_{\xi}|+|v_{\eta}|\lesssim
             |\mathbf{v}^{b}|_{1,K}$.
\end{remark}

 Define the modified partial derivatives $\frac{\tilde{\partial} v}{\partial x}$, $\frac{\tilde{\partial} v}{\partial y}$, the modified divergence  $\tilde{\mbox{div}}\mathbf{v}$ and the modified strain  $ \tilde{\varepsilon} (\mathbf{v})$ respectively as follows \cite{Zhang1997}: for $K\in T_{h},$
 \begin{equation*}
 %\label{modify1}
 (J_{K}\frac{\tilde{\partial} v}{\partial x}|_{K}\circ F_{K})(\xi,\eta)
 =\frac{\partial y}{\partial \eta}(0,0)\frac{\partial \hat{v}}{\partial \xi}
 -\frac{\partial y}{\partial \xi}(0,0)\frac{\partial \hat{v}}{\partial \eta}
 =b_{2}\frac{\partial \hat{v}}{\partial \xi}
 -b_{1}\frac{\partial \hat{v}}{\partial \eta},
 \end{equation*}
\begin{equation*}
%\label{modify2}
 (J_{K}\frac{\tilde{\partial} v}{\partial y}|_{K}\circ F_{K})(\xi,\eta)
 =-\frac{\partial x}{\partial \eta}(0,0)\frac{\partial \hat{v}}{\partial \xi}
 +\frac{\partial x}{\partial \xi}(0,0)\frac{\partial \hat{v}}{\partial \eta}
 =-a_{2}\frac{\partial \hat{v}}{\partial \xi}
 +a_{1}\frac{\partial \hat{v}}{\partial \eta},
 \end{equation*}
 \begin{equation*}
 %\label{modifydiv}
 \tilde{\mbox{div}} \mathbf{v}|_{K}
 =\frac{\tilde{\partial} u}{\partial x}
 +\frac{\tilde{\partial} v}{\partial y},
 \end{equation*}
\begin{equation*}
%\label{modifystrain}
 \tilde{\varepsilon} (\mathbf{v})|_{K}
 =\left(\begin{array}{cc}
 \frac{\tilde{\partial} u}{\partial x}
 &\frac{1}{2}(\frac{\tilde{\partial} u}{\partial y}
 +\frac{\tilde{\partial} v}{\partial x})\\
 \frac{1}{2}(\frac{\tilde{\partial} u}{\partial y}
 +\frac{\tilde{\partial} v}{\partial x})
 &\frac{\tilde{\partial} v}{\partial y}
\end{array}\right).
% \frac{1}{2}(\frac{\tilde{\partial} v_{1}}{\partial x}
% +\frac{\tilde{\partial} v_{2}}{\partial y}), \,i,j=1,2.
 \end{equation*}

 It is easy to verify that the PS stress mode satisfies the relation (see \cite{Piltner2000})
 \begin{equation}
 \label{PSconstraint}
 \int_{K}\boldsymbol{\tau}:\tilde{\varepsilon}(\mathbf{v}^{b})d\mathbf{x}=0\ \ \,\mbox{for all }
 \mathbf{v}^{b}\in B_{h}, 
 \end{equation}
 or equivalently
  \begin{equation*}
  \int_{K}(\boldsymbol{\tau}-\boldsymbol{\tau}_{0}):\varepsilon(\mathbf{v}^{b})d\mathbf{x}=0\ \ \,\mbox{for all }
 \mathbf{v}^{b}\in B_{h}
  \end{equation*}
for all $ \boldsymbol{\tau}\in \Sigma_{h}^{PS}$, with $\boldsymbol{\tau}_{0}$  the constant part of $\boldsymbol{\tau}$,
and that the ECQ4 stress mode satisfies  the so-called energy-compatibility condition (see \cite{Xie-Zhou2004,Zhou-Nie2001})
 \begin{equation}
 \label{ECconstraint}
 \int_{K}\boldsymbol{\tau}:\varepsilon(\mathbf{v}^{b})d\mathbf{x}=0\ \ \,\mbox{for all }
 \mathbf{v}^{b}\in B_{h}
 \end{equation}
for all $ \boldsymbol{\tau}\in \Sigma_{h}^{EC}$. As a result, the stress spaces $\Sigma_{h}^{PS}, \Sigma_{h}^{EC}$ can also be rewritten as
 \begin{eqnarray}\label{Sigma-PS}
\Sigma_{h}^{PS}
  &=&\{\boldsymbol{\tau}\in \Sigma:
  \hat{\boldsymbol{\tau}}_{ij}=\boldsymbol{\tau}_{ij}|_{K} \circ F_{K}\in P_{1}(\xi,\eta),
  \int_{K}\boldsymbol{\tau}:\tilde{\varepsilon}(\mathbf{v}^{b})d\mathbf{x}=0,\nonumber\\
 &&
  1\leq i\leq j\leq 2\ \ 
  \mbox{for all } \mathbf{v}^{b}\in B_{h}, \ K\in T_{h}\},\label{PSstress}
 \end{eqnarray}
\begin{eqnarray}\label{Sigma-EC}
\Sigma_{h}^{EC}
 &=&\{\boldsymbol{\tau}\in \Sigma:
 \hat{\boldsymbol{\tau}}_{ij}=\boldsymbol{\tau}_{ij}|_{K} \circ F_{K}\in  P_{1}(\xi,\eta),
 \int_{K}\boldsymbol{\tau}:\varepsilon(\mathbf{v}^{b})d\mathbf{x}=0,\nonumber\\
 && 1\leq i\leq j\leq 2\ \ 
 \mbox{for all } \mathbf{v}^{b}\in B_{h}, \ K\in T_{h}\}. \label{ECQ4stress}
 \end{eqnarray}

With the continuous isoparametric bilinear displacement approximation $V_{h}$ given in (\ref{displacement}),  the corresponding hybrid finite element schemes for  PS and ECQ4 are obtained by respectively taking $\Sigma_{h}=\Sigma_{h}^{PS}$ and  $\Sigma_{h}=\Sigma_{h}^{EC}$ in the discretized model (\ref{discreteweak1.a})(\ref{discreteweak1.b}).

\begin{remark} Since the stress approximation of the hybrid elements is piecewise-independent, the stress parameters,  $\beta^{\tau}$ in (\ref{PSstress1}) or (\ref{ECQ4stress1}), can be eliminated at the element level.  In this sense, the computational cost of the hybrid methods is almost the same as that of the isoparametric bilinear element.
\end{remark}

\subsection*{3.3. \ Numerical performance of hybrid elements}

 Three test problems are used to examine numerical
 performance of the hybrid elements PS/ECQ4.  The former two are benchmark tests widely used in literature, e.g. \cite{Pian-Sumihara1984, Pian-Wu1988,Xie-Zhou2004, Xie2005,Xie-Zhou2008,Zhou-Nie2001}, to test membrane elements while using coarse meshes, where no analytical forms of the exact solutions were given and numerical results were only computed at some special points. Here we give the explicit forms of the exact solutions and compute the stress error in $L^{2}$-norm and the displacement error in $H^{1}$-seminorm.   For comparison, the standard 4-node displacement element, i.e. the isoparametric bilinear element (abbr. bilinear), is also computed with  $5\times5$ Gaussian quadrature. For elements PS and ECQ4,  $2\times2$ Gaussian quadrature is exact in all the problems.

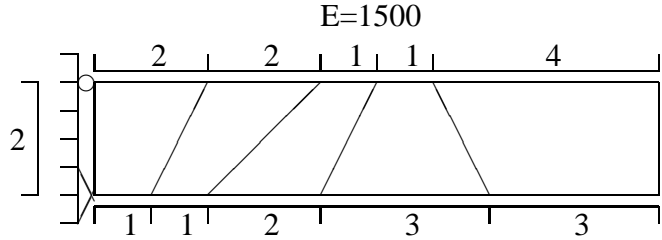
\begin{figure}[!h]
\begin{center}
\setlength{\unitlength}{0.75cm}
\begin{picture}(10,3)
\put(0,0){\line(1,0){10}} \put(0,0){\line(0,1){2}}
\put(0,2){\line(1,0){10}} \put(10,0){\line(0,1){2}}
\put(1,0){\line(1,2){1}} \put(2,0){\line(1,1){2}}
\put(4,0){\line(1,2){1}} \put(7,0){\line(-1,2){1}}

 \put(0,-0.2){\line(1,0){10}}
 \put(0,-0.2){\line(0,-1){0.3}}
 \put(1,-0.2){\line(0,-1){0.3}}
 \put(2,-0.2){\line(0,-1){0.3}}
 \put(4,-0.2){\line(0,-1){0.3}}
 \put(7,-0.2){\line(0,-1){0.3}}
 \put(10,-0.2){\line(0,-1){0.3}}

 \put(0.5,-0.7){1}
 \put(1.5,-0.7){1}
 \put(3,-0.7){2}
 \put(5.5,-0.7){3}
 \put(8.5,-0.7){3}

 \put(0,2.2){\line(1,0){10}}
 \put(0,2.2){\line(0,1){0.3}}
 \put(2,2.2){\line(0,1){0.3}}
 \put(4,2.2){\line(0,1){0.3}}
 \put(5,2.2){\line(0,1){0.3}}
 \put(6,2.2){\line(0,1){0.3}}
 \put(10,2.2){\line(0,1){0.3}}

 \put(1,2.3){2}
 \put(3,2.3){2}
 \put(4.5,2.3){1}
 \put(5.5,2.3){1}
 \put(8,2.3){4}

 \put(-0.3,-0.5){\line(0,1){3}}
 \multiput(-0.3,-0.5)(0,0.5){7}{\line(-1,0){0.3}}

 \put(-1,0){\line(0,1){2}}
 \put(-1,0){\line(-1,0){0.3}}
 \put(-1,2){\line(-1,0){0.3}}

 \put(-1.5,0.8){2}

 \put(-0.3,-0.5){\line(1,2){0.3}}
 \put(-0.3,0.5){\line(1,-2){0.3}}

 \put(-0.15,2){\circle{0.3}}

% \put(11.5,2){\vector(-1,0){1}}
% \put(10.5,0){\vector(1,0){1}}
% \put(10.5,1.5){1000}
% \put(10.5,0.2){1000}

 \put(4,3){E=1500}%
\end{picture}
\end{center}
\caption{Cantilever beam}
\end{figure}

\begin{figure}[!h]
\begin{center}
\setlength{\unitlength}{0.5cm}
\begin{picture}(22,7)
\put(0,3){\line(1,0){10}} \put(0,3){\line(0,1){2}}
\put(0,5){\line(1,0){10}} \put(10,3){\line(0,1){2}}
\put(2,3){\line(0,1){2}}  \put(4,3){\line(0,1){2}}
\put(6,3){\line(0,1){2}}  \put(8,3){\line(0,1){2}}
\put(4,2.4){$5\times 1$}

\put(0,0){\line(1,0){10}} \put(0,0){\line(0,1){2}}
\put(0,2){\line(1,0){10}} \put(10,0){\line(0,1){2}}
\put(1,0){\line(0,1){2}} \put(2,0){\line(0,1){2}}
\put(3,0){\line(0,1){2}} \put(4,0){\line(0,1){2}}
\put(5,0){\line(0,1){2}} \put(6,0){\line(0,1){2}}
\put(7,0){\line(0,1){2}} \put(8,0){\line(0,1){2}}
\put(9,0){\line(0,1){2}} \put(0,1){\line(1,0){10}}
\put(4,-0.6){$10\times 2$} \put(2,-2){regular meshes}

\put(15,3){\line(1,0){10}} \put(15,3){\line(0,1){2}}
\put(15,5){\line(1,0){10}} \put(25,3){\line(0,1){2}}
\put(16,3){\line(1,2){1}} \put(17,3){\line(1,1){2}}
\put(19,3){\line(1,2){1}} \put(22,3){\line(-1,2){1}}
\put(19,2.4){$5\times 1$}

\put(15,0){\line(1,0){10}} \put(15,0){\line(0,1){2}}
\put(15,2){\line(1,0){10}} \put(25,0){\line(0,1){2}}
\put(16,0){\line(1,2){1}} \put(17,0){\line(1,1){2}}
\put(19,0){\line(1,2){1}} \put(22,0){\line(-1,2){1}}
\put(15.5,0){\line(1,4){0.5}} \put(16.5,0){\line(3,4){1.5}}
\put(18,0){\line(3,4){1.5}}   \put(20.5,0){\line(0,1){2}}
\put(23.5,0){\line(-1,4){0.5}} \put(15,1){\line(1,0){10}}
\put(19,-0.6){$10\times 2$} \put(17,-2){irregular meshes}
\end{picture}
\end{center}
\vspace{0.5cm} \caption{Finite element meshes}
\end{figure}
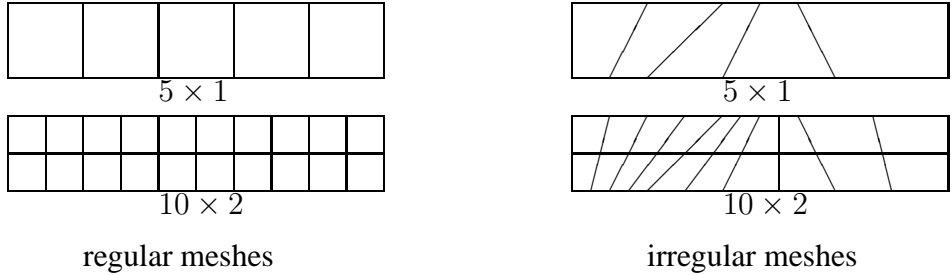

 \noindent
 {\bf Example 1. Beam bending test}

 A plane stress beam   modeled   with  different  meshes is computed (Figure 2 and Figure 3),
where the origin of the coordinates $x,y$ is at the midpoint of the left end, the body force $\mathbf{f}=(0,\ 0)^{T}$, 
 the surface traction $\mathbf{g}$ on $\Gamma_{N}=\{(x,y)\in [0,10]\times[-1,1]:\ x=10\mbox{or } y=\pm 1\}$ is given by 
 $\mathbf{g}|_{x=10}=(-2Ey,\ 0)^{T},$
$\mathbf{g}|_{y=\pm 1}=(0,\ 0)^{T}$,  and the exact solution is
 $$\mathbf{u}=\left(\begin{array}{c}
     -2xy\\
     x^{2}+\nu(y^{2}-1)\\
    \end{array}\right),\ \ \ \boldsymbol{\sigma}=\left(\begin{array}{cc}
     -2Ey& 0\\
      0&0\\
    \end{array}\right).$$
    
  The displacement and stress results, $\frac{|\mathbf{u}-\mathbf{u}_{h}|_{1}}{|\mathbf{u}|_{1}}$ and
 $\frac{\|\boldsymbol{\sigma}-\boldsymbol{\sigma}_{h}\|_{0}}{\|\boldsymbol{\sigma}\|_{0}}$, are listed respectively in Tables 1-2 with  $\nu=0.25$ and $E=1500$. Though of the same first-order convergence rate  in the displacement approximation, the hybrid elements results appear much more accurate when compared with the bilinear element.  Amazingly, the hybrid elements yield quite accurate stress results.

\begin{table}[!h]\renewcommand{\baselinestretch}{1.25}\small
 \centering
 \caption{The results of $\frac{|\mathbf{u}-\mathbf{u}_{h}|_{1}}{|\mathbf{u}|_{1}}$ in the plain stress beam  test}
\begin{tabular}{cccccccccc}
 \hline
 &&regular &mesh &  & &&irregular &mesh &\\
 \cline{2-5}\cline{7-10}
 method  &$5\times1$ &$10\times2$ &$20\times4$ &$40\times8$ & &$5\times1$ &$10\times2$ &$20\times4$ &$40\times8$\\\hline
 bilinear       &0.3256      &0.1106    &0.03376    &0.01165   & &0.5777    &0.2668     &0.09273    &0.02881\\
 PS       &0.07269     &0.03635   &0.01817    &0.009087  & &0.1429     &0.06303    &0.03113     &0.01552\\
 ECQ4     &0.07269     &0.03635   &0.01817    &0.009087  & &0.1313    &0.06256    &0.03107    &0.01551\\\hline
\end{tabular}
\end{table}
\begin{table}[!h]\renewcommand{\baselinestretch}{1.25}\small
 \centering
 \caption{The results of  $\frac{\|\boldsymbol{\sigma}-\boldsymbol{\sigma}_{h}\|_{0}}{\|\boldsymbol{\sigma}\|_{0}}$ in the plain stress beam  test }
\begin{tabular}{cccccccccc}
 \hline
  &&regular &mesh &  & &&irregular &mesh &\\
  \cline{2-5}\cline{7-10}
 method  &$5\times1$ &$10\times2$ &$20\times4$ &$40\times8$ & &$5\times1$ &$10\times2$ &$20\times4$ &$40\times8$\\\hline
biliear       &0.5062      &0.2951   &0.1545   &0.07826  & &0.7242     &0.4854     &0.2809     &0.1481\\
 PS       &0     &0  &0  &0        & &0.2663     &0.05559     &0.01134    &0.002551\\
 ECQ4     &0     &0  &0  &0  & &0.1780    &0.03517     &0.007324   &0.001666\\\hline
\end{tabular}
\end{table}

\noindent
 {\bf Example 2.  Poisson's ratio locking-free test}

 A plane strain pure bending cantilever beam is used to test
 locking-free performance,  with the same domain and meshes as in Figures 2 and 3.  In this case, the body force $\mathbf{f}=(0,\ 0)^{T}$,
  the surface traction $\mathbf{g}$ on $\Gamma_{N}=\{(x,y)\in [0,10]\times[-1,1]:\ x=10\mbox{or } y=\pm 1\}$ is given by 
 $\mathbf{g}|_{x=10}=(-2Ey,\ 0)^{T},$  $\mathbf{g}|_{y=\pm1}=(0,\ 0)^{T},$
 and the exact solution is
 $$\mathbf{u}=\left(\begin{array}{c}
     -2(1-\nu^{2})xy\\
     (1-\nu^{2})x^{2}+\nu(1+\nu)(y^{2}-1)\\
    \end{array}\right), \ \ \ \boldsymbol{\sigma}=\left(\begin{array}{cc}
     -2Ey&0\\
      0&0\\
    \end{array}\right).$$
 
 The numerical results with $E=1500$ and different values of  Poisson ratio $\nu$ are listed in Tables 3-7. As we can see, the bilinear element deteriorates  as $\nu\rightarrow 0.5$ or $\lambda\rightarrow \infty$, whereas the two hybrid elements give uniformly good results, with first order   accuracy  for the displacement approximation in $H^{1}$-seminorm  and second order accuracy for the stress in $L^{2}$-norm.

\begin{table}[!h]\renewcommand{\baselinestretch}{1.25}\small
 \centering
 \caption{The results of $\frac{|\mathbf{u}-\mathbf{u}_{h}|_{1}}{|\mathbf{u}|_{1}}$ for the bilinear element in the plane strain  test}
\begin{tabular}{cccccccccc}
 \hline
  &&regular &mesh &  & &&irregular &mesh &\\\cline{2-5}\cline{7-10}
 $\nu$  &$5\times1$ &$10\times2$ &$20\times4$ &$40\times8$ & &$5\times1$ &$10\times2$ &$20\times4$ &$40\times8$\\\hline
 0.49         &0.9253      &0.7547      &0.4353       &0.1620 & &0.8862     &0.7641      &0.5351     &0.2597\\
 0.499        &0.9921      &0.9690      &0.8866      &0.6619 & &0.9515     &0.9241      &0.8530     &0.6978\\
 0.4999       &0.9992      &0.9968      &0.9874      &0.9514 & &0.9615      &0.9567      &0.9446     &0.9067\\
 0.49999      &0.9999      &0.9997      &0.9987      &0.9949 & &0.9626      &0.9606      &0.9591     &0.9540\\\hline
% 0.4999999    &1            &1            &0.9999      &0.9999 & &0.9627      &0.9611      &0.9609     &0.9608\\
% 0.4999999999 &1            &1            &1            &1       & &0.96267      &0.96106      &0.9609      &0.96086\\
%\hline
\end{tabular}
\end{table}

\begin{table}[!h]\renewcommand{\baselinestretch}{1.25}\small
 \centering
 \caption{The results of $\frac{|\mathbf{u}-\mathbf{u}_{h}|_{1}}{|\mathbf{u}|_{1}}$ for PS in the plane strain test}
\begin{tabular}{cccccccccc}
 \hline
  &&regular &mesh &  & &&irregular &mesh &\\\cline{2-5}\cline{7-10}
 $\nu$  &$5\times1$ &$10\times2$ &$20\times4$ &$40\times8$ & &$5\times1$ &$10\times2$ &$20\times4$ &$40\times8$\\\hline
 0.49         &0.09759     &0.04879     &0.02440     &0.01220 & &0.1557     &0.07342     &0.03649     &0.01822\\
 0.499        &0.09931     &0.04965     &0.02483     &0.01241 & &0.1567     &0.07410     &0.03684     &0.01839\\
 0.4999       &0.09948     &0.04974     &0.02487     &0.01244 & &0.1569     &0.07418     &0.03688     &0.01841\\
 0.49999      &0.09950     &0.04975     &0.02488     &0.01244 & &0.1569     &0.07418     &0.03688     &0.01841\\
 % 0.4999999    &0.09950     &0.04975     &0.02488     &0.01244 & &0.1569     &0.07418     &0.03689     &0.01842\\
% 0.4999999999 &0.099504     &0.049752     &0.024877     &0.012438 & &0.15687     &0.074183     &0.036883     &0.018467\\
\hline
\end{tabular}
\end{table}

\begin{table}[!h]\renewcommand{\baselinestretch}{1.25}\small
 \centering
 \caption{The results of $\frac{\|\boldsymbol{\sigma}-\boldsymbol{\sigma}_{h}\|_{0}}{\|\boldsymbol{\sigma}\|_{0}}$ for PS in the plane strain test}
\begin{tabular}{cccccccccc}
 \hline
  &&regular &mesh &  & &&irregular &mesh &\\\cline{2-5}\cline{7-10}
 $\nu$  &$5\times1$ &$10\times2$ &$20\times4$ &$40\times8$ & &$5\times1$ &$10\times2$ &$20\times4$ &$40\times8$\\\hline
 0.49         &0  &0  &0 &0  & &0.2286      &0.04566    &0.009326    &0.002094\\
 0.499        &0  &0  &0  &0  & &0.2268      &0.0452    &0.009238    &0.002073\\
 0.4999       &0  &0  &0  &0  & &0.2266      &0.04516   &0.009229    &0.002071\\
 0.49999      &0  &0  &0  &0  & &0.2266      &0.04516    &0.009229    &0.002071\\
% 0.4999999    &0  &3.64e-05  &6.32e-05  &2.55e-04  & &0.22655      &0.046764   &0.080129     &0.64946\\
%c 0.4999999999 &2.77e-05  &36.66     &39.13    &228.36    & &0.22655      &31686      &88479        &5.8854e+05\\
\hline
\end{tabular}
\end{table}

\begin{table}[!h]\renewcommand{\baselinestretch}{1.25}\small
 \centering
 \caption{The results of $\frac{|\mathbf{u}-\mathbf{u}_{h}|_{1}}{|\mathbf{u}|_{1}}$ for ECQ4 in the plane strain  test}
\begin{tabular}{cccccccccc}
 \hline
  &&regular &mesh &  & &&irregular &mesh &\\\cline{2-5}\cline{7-10}
 $\nu$  &$5\times1$ &$10\times2$ &$20\times4$ &$40\times8$ & &$5\times1$ &$10\times2$ &$20\times4$ &$40\times8$\\\hline
 0.49         &0.09759     &0.04879     &0.02440     &0.01220 & &0.1512     &0.07321     &0.03647     &0.01821\\
 0.499        &0.09931     &0.04965     &0.02483     &0.01241 & &0.1526     &0.07392     &0.03682     &0.01839\\
 0.4999       &0.09948     &0.04974     &0.02487     &0.01244 & &0.1527     &0.07399     &0.03686     &0.01841\\
 0.49999      &0.09950     &0.04975     &0.02488     &0.01244 & &0.1569     &0.07418     &0.03688     &0.01841\\
 % 0.4999999    &0.099504     &0.049752     &0.024876     &0.012438 & &0.15273     &0.073997     &0.036862     &0.018412\\
% 0.4999999999 &0.099504     &0.049752     &0.024877     &0.012438 & &0.15273     &0.073997     &0.036862     &0.018412
\hline
\end{tabular}
\end{table}

\begin{table}[!h]\renewcommand{\baselinestretch}{1.25}\small
 \centering
 \caption{The results of $\frac{\|\boldsymbol{\sigma}-\boldsymbol{\sigma}_{h}\|_{0}}{\|\boldsymbol{\sigma}\|_{0}}$ for ECQ4 in the plane strain test}
\begin{tabular}{cccccccccc}
 \hline
  &&regular &mesh &  & &&irregular &mesh &\\\cline{2-5}\cline{7-10}
 $\nu$  &$5\times1$ &$10\times2$ &$20\times4$ &$40\times8$ & &$5\times1$ &$10\times2$ &$20\times4$ &$40\times8$\\\hline
 0.49         &0  &0 &0  &0  & &0.1780     &0.03456    &0.007270    &0.001661\\
 0.499        &0  &0 &0  &0  & &0.1780     &0.03455    &0.007274    &0.001662\\
 0.4999       &0  &0 &0  &0  & &0.1780     &0.03455    &0.007275    &0.001662\\
 0.49999      &0  &0 &0  &0  & &0.1780     &0.03455    &0.007275    &0.001662\\
% 0.4999999    &1.35e-08  &3.64e-05 &6.32e-05  &2.55e-04  & &0.178     &0.034551    &0.0072751    &0.0016619\\
% 0.4999999999 &2.77e-05  &36.66    &39.13     &228.36    & &0.178     &0.034551    &0.0072751    &0.0016619
\hline
\end{tabular}
\end{table}

\noindent
 {\bf Example 3.  A new plane stress test}

 In the latter two tests, the hybrid elements give quite accurate numerical results for the stress approximation. This is partially owing  to the fact that the analytical stress solutions are linear polynomials in both cases.  To verify this, we compute a new plane stress test with the same domain and meshes as in Figures 2 and 3. Here
 the body force has the form 
 $\mathbf{f}=-\ (  6y^{2},\  6x^{2})^{T}$,
 the surface traction $\mathbf{g}$ on$\Gamma_{N}=\{(x,y):\ x=10,\ -1\leq y\leq1\}$
is given by
 $\mathbf{g}=\ (0,\  2000+2y^{3})^{T}$, 
 and the exact solution is
 $$\mathbf{u}=\frac{\nu+1}{E}( y^{4},\ x^{4})^{T},\ \ \ \boldsymbol{\sigma}=
      \left(\begin{array}{cc}
      0 &2(x^{3}+y^{3})\\
      2(x^{3}+y^{3}) &0\\
    \end{array}\right).$$
 We only compute the the case of $E=1500,\ \nu=0.25$ for PS and ECQ4 and list the results in Tables
 8-9. It is easy to see that  the displacement accuracy in $H^{1}-$seminorm, as well as the stress accuracy in
 $L^{2}$-norm, is of order 1.

 \begin{table}[!h]\renewcommand{\baselinestretch}{1.25}\small
 \centering
 \caption{The error $\frac{|\mathbf{u}-\mathbf{u}_{h}|_{1}}{|\mathbf{u}|_{1}}$ of Example 3}
\begin{tabular}{cccccccccc}
 \hline
  &&regular &mesh &  & &&irregular &mesh &\\\cline{2-5}\cline{7-10}
 method  &$10\times2$ &$20\times4$ &$40\times8$ &$80\times16$& &$10\times2$ &$20\times4$ &$40\times8$ &$80\times16$ \\\hline
 PS       &0.1022     &0.05120    &0.02561     &0.01281  & &0.1815     &0.08968     &0.04470     &0.02233\\
 ECQ4     &0.1022     &0.05120     &0.02561     &0.01281  & &0.1815     &0.08968     &0.04470     &0.02233\\\hline
\end{tabular}
\end{table}
\begin{table}[!h]\renewcommand{\baselinestretch}{1.25}\small
 \centering
 \caption{The error $\frac{\|\boldsymbol{\sigma}-\boldsymbol{\sigma}_{h}\|_{0}}{\|\boldsymbol{\sigma}\|_{0}}$ of Example 3}
\begin{tabular}{cccccccccc}
 \hline
  &&regular &mesh &  & &&irregular &mesh &\\\cline{2-5}\cline{7-10}
 method  &$10\times2$ &$20\times4$ &$40\times8$ &$80\times16$ & &$10\times2$ &$20\times4$ &$40\times8$ &$80\times16$ \\\hline
% Q4       &0.10219     &0.051201     &0.025614     &0.012809  & &0.20447     &0.099464     &0.049464     &0.024703\\
 PS       &0.1022     &0.05120     &0.02561     &0.01281  & &0.1806     &0.08590     &0.04239     &0.02113\\
 ECQ4     &0.1022     &0.05120     &0.02561     &0.01281  & &0.1850     &0.09103     &0.04532    &0.02264\\\hline
\end{tabular}
\end{table}

 \setcounter{remark}{0}\setcounter{lemma}{0} \setcounter{theorem}{0}
 \setcounter{section}{4} \setcounter{equation}{0}
\section*{4. \ Uniform a priori error estimates }

\subsection*{4.1. Error analysis for the PS finite element}

To derive  uniform error estimates for the hybrid methods, according to the mixed method theory   \cite{Brezzi1974,Brezzi-Fortin1991}, we need  the following two discrete versions of the stability conditions (A1) and (A2):

 \noindent  ($\mathrm{A1}_{h}$) \ Discrete Kernel-coercivity: For any $\boldsymbol{\tau}\in Z_{h}:=\{\boldsymbol{\tau}\in \Sigma_{h}:
 \int_{\Omega}\boldsymbol{\tau}:\varepsilon(\mathbf{v})d\mathbf{x}=0, \mbox{for all } \mathbf{v}\in V_{h}\}$,
 it holds
 $$ \|\boldsymbol{\tau}\|_{0}^{2}\lesssim a(\boldsymbol{\tau},\boldsymbol{\tau});$$
 ($\mathrm{A2}_{h}$) \ Discrete inf-sup condition: For any $\mathbf{v}\in V_{h}$, it holds
 $$  |\mathbf{v}|_{1}\lesssim\sup_{0\neq\boldsymbol{\tau}\in \Sigma_{h}}
 \frac{\int_{\Omega}\boldsymbol{\tau}:\varepsilon(\mathbf{v})d\mathbf{x}}{\|\boldsymbol{\tau}\|_{0}} .$$

Introduce the spaces
 \begin{equation*}
  W_{h}:=\{q\in L^{2}(\Omega): \hat{q}=q|_{K}\circ F_{K}\in P_{1}(\hat{K})\ \ \mbox{for all } K\in
  T_{h}\},
 \end{equation*}
 \begin{equation*}
  \bar{W}_{h}:=\{\bar{q}\in W_{h}: \bar{q}|_{K}\in P_{0}(K)\ \ \mbox{for all } K\in
  T_{h}\}.
 \end{equation*}
To prove the stability condition $(A1_{h})$ for the PS finite element, we need   the following lemma.

\begin{lemma}
 (\cite{Zhang1997})
 Let the partition $T_{h}$ satisfy the shape-regularity condition (\ref{partition condition}).
 Assume that for any $\bar{q}\in \bar{W}_{h}$,
 there exists some $\mathbf{v}\in V_{h}$ with
 \begin{equation}\label{Q1-P0infsup1}
\|\bar{q}\|_{0}^{2}\lesssim \int_{\Omega}\bar{q}\ {\bf div} \mathbf{v}d\mathbf{x},\ \ |\mathbf{v}|_{1}^{2}\lesssim\|\bar{q}\|_{0}^{2}.
 \end{equation}
 Then it holds
\begin{equation}
\|q\|_{0} \lesssim \sup _{\mathbf{v}\in V_{h},\mathbf{v}^{b}\in B_{h}}\frac{\int_{\Omega}q ({\bf div} \mathbf{v}+\tilde{{\bf div}}
 \mathbf{v}^{b})d\mathbf{x}}{|\mathbf{v}+\mathbf{v}^{b}|_{1,h}}\ \ \ \mbox{ for all }
 q\in W_{h},
\end{equation}
 where  the semi-norm $|\cdot|_{1,h}$ on $V_{h}+B_{h}$ is defined as
 $
 |\cdot|_{1,h}:=(\sum_{K\in
 T_{h}}|\cdot|_{1,K}^{2})^{\frac{1}{2}}. 
 $
\end{lemma}
\begin{remark} Under the shape-regularity condition (\ref{partition
 condition}),  the following special property has been shown in \cite{Zhang1997}:
 $$ |\mathbf{v}|_{1}+|\mathbf{v}^{b}|_{1,h}\lesssim |\mathbf{v}+\mathbf{v}^{b}|_{1,h}\lesssim|\mathbf{v}|_{1}+|\mathbf{v}^{b}|_{1,h}\ \ \ \mbox{for all } \mathbf{v}\in V_{h},\mathbf{v}^{b}\in B_{h}.$$
 \end{remark}

In view of this lemma, we have 

\begin{theorem}
 Under the same conditions as in Lemma 4.1,  the uniform discrete Kernel-coercivity condition ($\mathrm{A1}_{h}$)
 holds for the PS finite element with $\boldsymbol{\sigma}_{h}=\boldsymbol{\sigma}_{h}^{PS}.$
\end{theorem}
\begin{proof}
 Similar to the proof of Theorem 2.1,  it suffices to show $\|tr\boldsymbol{\tau}\|_{0}\lesssim\|\boldsymbol{\tau}^{D}\|_{0}$ for any $\boldsymbol{\tau}\in Z_{h}$.
 
 In fact,  for $\boldsymbol{\tau}\in Z_{h}$, $\mbox{for all }\mathbf{v}\in V_{h}$ and $\mbox{for all } \mathbf{v}^{b}\in
 B_{h}$, it holds
\begin{eqnarray*}
 0 &= &\int_{\Omega}\boldsymbol{\tau}:\varepsilon(\mathbf{v})d\mathbf{x}\\
   &=
   &\int_{\Omega}\boldsymbol{\tau}:(\varepsilon(\mathbf{v})+\tilde{\varepsilon}(\mathbf{v}^{b}))d\mathbf{x}\\
   &= &\int_{\Omega}(\frac{1}{2}tr\boldsymbol{\tau}\mathbf{I}+\boldsymbol{\tau}^{D}):
    (\varepsilon(\mathbf{v})+\tilde{\varepsilon}(\mathbf{v}^{b}))d\mathbf{x}\\
   &= &\int_{\Omega}\frac{1}{2}tr\boldsymbol{\tau}(\mbox{div}\mathbf{v}+\tilde{\mbox{div}}\mathbf{v}^{b})d\mathbf{x}
   +\int_{\Omega}\boldsymbol{\tau}^{D}:(\varepsilon(\mathbf{v})+\tilde{\varepsilon}(\mathbf{v}^{b}))d\mathbf{x}.
\end{eqnarray*}
 Thus, by Lemma 4.1, we get {\footnotesize
\begin{eqnarray*}
\|tr\boldsymbol{\tau}\|_{0}&\lesssim& \sup\limits_{\mathbf{v}\in V_{h},\mathbf{v}^{b}\in B_{h}}
 \frac{\int_{\Omega}tr\boldsymbol{\tau}(\mbox{div}\mathbf{v}+\tilde{\mbox{div}}\mathbf{v}^{b})d\mathbf{x}}{|\mathbf{v}+\mathbf{v}^{b}|_{1,h}}\\
&=& \sup\limits_{\mathbf{v}\in V_{h},\mathbf{v}^{b}\in B_{h}}
 \frac{-2\int_{\Omega}\boldsymbol{\tau}^{D}:(\varepsilon(\mathbf{v})+\tilde{\varepsilon}(\mathbf{v}^{b}))d\mathbf{x}}{|\mathbf{v}+\mathbf{v}^{b}|_{1,h}}
 \lesssim \|\boldsymbol{\tau}^{D}\|_{0}.
\end{eqnarray*}
}
This completes the proof.
 \end{proof}

 This theorem states that any quadrilateral mesh which is stable for
 the Stokes element Q1-P0  is sufficient for ($\mathrm{A1}_{h}$). As we know,   the only unstable case for  Q1-P0   is the
 checkerboard mode.  Thereupon,  any quadrilateral mesh which
 breaks the checkerboard mode is sufficient for the uniform stability
 ($\mathrm{A1}_{h}$).

The latter part of this subsection is devoted to the proof of the discrete inf-sup condition ($\mathrm{A2}_{h}$).  It should be pointed out that in \cite{Zhou-Xie2002} there has been a proof for this stability condition. However, we shall  give a more simpler one here.  
%********************

 From (\ref{displacementform}), for any $\mathbf{v}\in V_{h}$ we have
{\small  $$
 J_{K}
 \left(\begin{array}{c}
  \frac{\partial u}{\partial x}\\
  \frac{\partial v}{\partial y}\\
  \frac{\partial u}{\partial y}+\frac{\partial v}{\partial x}\\
 \end{array}\right)=
 \left(\begin{array}{c}
  (U_{1}b_{2}-U_{2}b_{1})+(U_{1}b_{12}-U_{12}b_{1})\xi+(U_{12}b_{2}-U_{2}b_{12})\eta\\\\
  (V_{2}a_{1}-V_{1}a_{2})+(V_{12}a_{1}-V_{1}a_{12})\xi+(V_{2}a_{12}-V_{12}a_{2})\eta\\\\
  (U_{2}a_{1}-U_{1}a_{2})+(U_{12}a_{1}-U_{1}a_{12})\xi+(U_{2}a_{12}-U_{12}a_{2})\eta\\
  +(V_{1}b_{2}-V_{2}b_{1})+(V_{1}b_{12}-V_{12}b_{1})\xi+(V_{12}b_{2}-V_{2}b_{12})\eta\\
 \end{array}\right)
 $$
\begin{equation}\label{epsilon-new}
 = 
 \left(\begin{array}{ccccc}
  b_{2}+b_{12}\xi &-b_{1}-b_{12}\eta &-b_{1}\xi+b_{2}\eta &0 &0\\
  0 &0 &0 &a_{1}+a_{12}\eta &a_{1}\xi-a_{2}\eta\\
  -a_{2}-a_{12}\xi &a_{1}+a_{12}\eta &a_{1}\xi-a_{2}\eta
  &-b_{1}-b_{12}\eta &-b_{1}\xi+b_{2}\eta
 \end{array}\right)\beta^v
 \end{equation}
 }
 with $\beta^v=(\beta^v_1,\cdots,\beta_5^v)^T:={\small (
 U_{1}+\frac{b_{1}}{a_{1}}V_{1},
U_{2}+\frac{b_{2}}{a_{1}}V_{1}, U_{12}+\frac{b_{12}}{a_{1}}V_{1},
V_{2}-\frac{a_{2}}{a_{1}}V_{1},
V_{12}-\frac{a_{12}}{a_{1}}V_{1})^{T}.}
$

\begin{lemma}\label{epsilon-v}
For any $\mathbf{v}\in V_{h}$ and $K\in T_h$, it holds
\begin{equation}
\|\varepsilon(\mathbf{v})\|_{0,K}^{2}\lesssim
\frac{1}{\min\limits_{(\xi,\eta)\in \hat{K}}
J_{K}(\xi,\eta)}h_{K}^{2}\sum_{1\leq i\leq 5} (\beta^{v}_{i})^{2}.
\end{equation}
\end{lemma}

\begin{proof} From (\ref{epsilon-new}) we have
{\small
\begin{eqnarray*}
&&\|\varepsilon(\mathbf{v})\|_{0,K}^{2}=\int_{K}\varepsilon(\mathbf{v}):\varepsilon(\mathbf{v})dx\\
&=&\int_{\hat{K}}\left[((b_{2}+b_{12}\xi)\beta^{v}_{1}-(b_{1}+b_{12}\eta)\beta^{v}_{2}-(b_{1}\xi-b_{2}\eta)\beta^{v}_{3})^{2}
+((a_{1}+a_{12}\eta)\beta^{v}_{4}+(a_{1}\xi-a_{2}\eta)\beta^{v}_{5})^{2}\right.\\
&&\left.+\frac{1}{2}(-(a_{2}+a_{12}\xi)\beta^{v}_{1}+(a_{1}+a_{12}\eta)\beta^{v}_{2}+(a_{1}\xi-a_{2}\eta)\beta^{v}_{3}-(b_{1}+b_{12}\eta)\beta^{v}_{4}-(b_{1}\xi-b_{2}\eta)\beta^{v}_{5})^{2}\right]\\
&&\cdot J^{-1}_{K}(\xi,\eta)d\xi
d\eta\\
&\lesssim& \frac{1}{\min\limits_{(\xi,\eta)\in \hat{K}}
J_{K}(\xi,\eta)}h_{K}^{2}\sum_{1\leq i\leq 5} (\beta^{v}_{i})^{2}.
\end{eqnarray*}
}
\end{proof}

\begin{lemma}\label{PS-stress-ineq}
For any $\boldsymbol{\tau}\in\Sigma_{h}^{PS}$ and $K\in T_h$, it holds
\begin{equation}
\|\boldsymbol{\tau}\|_{0,K}^{2}\gtrsim \min_{(\xi,\eta)\in \hat{K}}
J_{K}(\xi,\eta)\sum_{1\leq i\leq 5} (\beta^{\tau}_{i})^{2}.
\end{equation}
\end{lemma}
\begin{proof} The form (\ref{PSstress1}) indicates
\begin{eqnarray*}
\|\boldsymbol{\tau}\|_{0,K}^{2}&=&\int_{K}\boldsymbol{\tau}:\boldsymbol{\tau} dx=\int_{\hat{K}}\left[(\beta^{\tau}_{1}+\eta \beta^{\tau}_{4}+\frac{a_{2}^{2}}{b_{2}^{2}}\xi \beta^{\tau}_{5})^{2}+(\beta^{\tau}_{2}+\frac{b_{1}^{2}}{a_{1}^{2}}\eta \beta^{\tau}_{4}+\xi \beta^{\tau}_{5})^{2}\right.\\
&&\hskip2cm\left.+2(\beta^{\tau}_{3}+\frac{b_{1}}{a_{1}}\eta
\beta^{\tau}_{4}+\frac{a_{2}}{b_{2}}\xi
\beta^{\tau}_{5})^{2}\right]J_{K}(\xi,\eta) d\xi
d\eta\\
&\geq &\frac{4}{3}\min\limits_{(\xi,\eta)\in \hat{K}}
J_{K}(\xi,\eta)\sum_{1\leq i\leq 5} (\beta^{\tau}_{i})^{2}.
\end{eqnarray*}
\end{proof}

\begin{lemma}\label{tau-epsilon-relation}
For any $\mathbf{v}\in V_{h}$, there exists a $\boldsymbol{\tau}_{v}\in \Sigma_{h}^{PS}$ such
that  for any   $K\in T_h$,
 \begin{equation}\label{tau-eps2}
 \int_{K}\boldsymbol{\tau}_{v}:\varepsilon(\mathbf{v}) dx = \|\boldsymbol{\tau}_{v}\|_{0,K}^{2} \gtrsim
 \|\varepsilon(\mathbf{v}) \|_{0,K}^{2}.
 \end{equation}

\end{lemma}

\begin{proof} We follow the same line as in the proof of [Lemma 4.4, \cite{Carstensen.C;Xie.X;Yu.G;Zhou2010}].

For $\boldsymbol{\tau}\in \Sigma_{h}^{PS}$ and $\mathbf{v}\in V_{h}$, from (\ref{PSstress1}) and (\ref{epsilon-new}) it holds
\begin{equation*}{\small
\int_{K}\boldsymbol{\tau}:\varepsilon(\mathbf{v})dx =(\beta^{\tau})^{T} \left(
\begin{array}{ccccc}
 4b_{2} &-4b_{1} &0 &0 &0 \\
 0 &0 &0 &4a_{1} &0 \\
 -4a_{2} &4a_{1} &0 &-4b_{1} &0\\
 0 &-\frac{4}{3}\frac{J_{1}}{a_{1}} &\frac{4}{3}\frac{J_{0}}{a_{1}} &-\frac{4}{3}\frac{b_{1}J_{1}}{a_{1}^{2}} &\frac{4}{3}\frac{b_{1}J_{0}}{a_{1}^{2}}\\
 -\frac{4}{3}\frac{a_{2}J_{2}}{b_{2}^{2}} &0 &\frac{4}{3}\frac{a_{2}J_{0}}{b_{2}^{2}}  &0 &\frac{4}{3}\frac{J_{0}}{b_{2}}\\
\end{array}\right)
\beta^{v}:=(\beta^{\tau})^{T}{\bf A}\beta^{v}.
}
\end{equation*} 
By  mean value theorem,  there exists a point $(\xi_{0},\eta_{0})\in [-1,1]^{2}$ such
that
\begin{equation}{\small
 \|\boldsymbol{\tau}\|_{0,K}^{2}= J_{K}(\xi_{0},\eta_{0})(\beta^{\tau})^{T}
{\bf D}
\beta^{\tau}
}
\end{equation}
with ${\bf D}=\mbox{diag}\left(4,4,8,\frac{4}{3}[1+2(\frac{b_{1}}{a_{1}})^{2}+(\frac{b_{1}^{2}}{a_{1}^{2}})^{2}],\frac{4}{3}[1+2(\frac{a_{2}}{b_{2}})^{2}+(\frac{a_{2}^{2}}{b_{2}^{2}})^{2}]\right)$.

Denote ${\bf \tilde{D}}:=\mbox{diag}\left(1,1,1,\frac{a_{1}^{4}}{a_{1}^{4}+2a_{1}^{2}b_{1}^{2}+b_{1}^{4}},\frac{b_{2}^{4}}{a_{2}^{4}+2a_{2}^{2}b_{2}^{2}+b_{2}^{4}}\right)$,
$$ {\bf \tilde{A}}:= \left(
\begin{array}{ccccc}
 b_{2} &-b_{1} &0 &0 &0 \\
 0 &0 &0 &a_{1} &0 \\
 -\frac{a_{2}}{2} &\frac{a_{1}}{2} &0 &-\frac{b_{1}}{2} &0\\
 0 &-\frac{J_{1}}{a_{1}} &\frac{J_{0}}{a_{1}} &-\frac{b_{1}J_{1}}{a_{1}^{2}} &\frac{b_{1}J_{0}}{a_{1}^{2}}\\
 -\frac{a_{2}J_{2}}{b_{2}^{2}} &0 &\frac{a_{2}J_{0}}{b_{2}^{2}}  &0 &\frac{J_{0}}{b_{2}}\\
\end{array}\right),$$
and take
 \begin{eqnarray*}
\boldsymbol{\tau}_{v}=\left(
\begin{array}{ccccc}
 1 &0 &0 &\eta &\frac{a_{2}^{2}}{b_{2}^{2}}\xi\\
 0 &1 &0 &\frac{b_{1}^{2}}{a_{1}^{2}}\eta &\xi\\
 0 &0 &1 &\frac{b_{1}}{a_{1}}\eta &\frac{a_{2}}{b_{2}}\xi
\end{array}\right)
\beta^{\tau,v}
\end{eqnarray*}
with
\begin{equation}\label{beta-tau-v}
\beta^{\tau, v}
=\frac{1}{J_{K}(\xi_{0},\eta_{0})}{\bf D}^{-1}{\bf A}\beta^{v}= \frac{1}{J_{K}(\xi_{0},\eta_{0})}{\bf \tilde{D}}{\bf \tilde{A}}\beta^{v},%=:{\bf z}=(z_1,z_2,\cdots,z_5)^T
\end{equation}
we then obtain
\begin{equation}\label{tau-eps}
\int_{K}\boldsymbol{\tau}_{v}:\varepsilon(\mathbf{v})dx=\|\boldsymbol{\tau}_{v}\|_{0,K}^{2}.
\end{equation}

On the other hand, (\ref{beta-tau-v}) yields
$$
\beta^{v}={J_{K}(\xi_{0},\eta_{0})}{\bf \tilde{A}}^{-1}{\bf \tilde{D}}^{-1}\beta^{\tau, v}
$$ 
with
$${\bf \tilde{A}}^{-1}=
 \left(
\begin{array}{ccccc}
  \frac{a_{1}}{J_{0}}
  &\frac{b_{1}^{2}}{a_{1}J_{0}}
  &\frac{2b_{1}}{J_{0}}
  &0
  &0\\
  \frac{a_{2}}{J_{0}}
  &\frac{b_{1}b_{2}}{a_{1}J_{0}}
  &\frac{2b_{2}}{J_{0}}
  &0
  &0\\
 \frac{a_{1}a_{2}(b_{2}J_{1}-b_{1}J_{2})}{J_{0}^{3}}
 &\frac{2b_{1}b_{2}^{2}J_{1}}{J_{0}^{3}}-\frac{a_{2}b_{1}^{2}b_{2}J_{1}}{a_{1}J_{0}^{3}}-\frac{a_{2}b_{1}^{3}J_{2}}{a_{1}J_{0}^{3}}
  &\frac{2(a_{1}b_{2}^{2}J_{1}-a_{2}b_{1}^{2}J_{2})}{J_{0}^{3}}
 &\frac{a_{1}^{2}b_{2}}{J_{0}^{2}}
 &\frac{-b_{1}b_{2}^{2}}{J_{0}^{2}}\\
 0
 &\frac{1}{a_{1}}
 &0
 &0
 &0\\
 \frac{a_{1}a_{2}(-a_{2}J_{1}+a_{1}J_{2})}{J_{0}^{3}}
 &\frac{-2a_{2}b_{1}b_{2}J_{1}}{J_{0}^{3}}+\frac{a_{2}^{2}b_{1}^{2}J_{1}}{a_{1}J_{0}^{3}}+\frac{a_{2}b_{1}^{2}J_{2}}{J_{0}^{3}}
 &\frac{2a_{1}a_{2}(-b_{2}J_{1}+b_{1}J_{2})}{J_{0}^{3}}
 &\frac{-a_{1}^{2}a_{2}}{J_{0}^{2}}
 &\frac{a_{1}b_{2}^{2}}{J_{0}^{2}}\\
\end{array}\right)$$
and ${\bf \tilde{D}}^{-1}=\mbox{diag}\left(
1,1,1, \frac{a_{1}^{4}+2a_{1}^{2}b_{1}^{2}+b_{1}^{4}}{a_{1}^{4}},\frac{a_{2}^{4}+2a_{2}^{2}b_{2}^{2}+b_{2}^{4}}{b_{2}^{4}}\right)$.
This relation, together with Lemma \ref{zhang}, (\ref{a1b2-hK}) and (\ref{JK-hK}), imply
\begin{equation*}
\sum_{1\leq i\leq 
5}(\beta^{v}_{i})^{2}\lesssim  h_K^2\sum_{1\leq i\leq
5}(\beta^{\tau,v}_{i})^{2}.
\end{equation*}
Combining  this inequality with  Lemmas \ref{epsilon-v}-\ref{PS-stress-ineq} and (\ref{JK-hK}), we arrive at
\begin{eqnarray*}
\|\boldsymbol{\tau}_{v}\|_{0,K}^{2}\gtrsim \|\varepsilon(\mathbf{v})\|_{0,K}^{2}.
\end{eqnarray*}
This inequality, together (\ref{tau-eps}), shows the conclusion.
\end{proof}

%**********************

\begin{theorem}
 Let the partition $T_{h}$ satisfy the shape-regularity condition (\ref{partition condition}). Then
 the uniform discrete inf-sup condition (${A2}_{h}$) holds with 
 $\Sigma_{h}= \Sigma_{h}^{PS}.$
\end{theorem}

\begin{proof} From Lemma \ref{tau-epsilon-relation}, 
 for any $\mathbf{v}\in V_{h}$,
 there exists $\boldsymbol{\tau}_v\in \Sigma_{h}^{PS}$ such that (\ref{tau-eps2}) holds. Then it holds
 \begin{eqnarray*}
\|\boldsymbol{\tau}_{v}\|_{0}|\mathbf{v}|_{1}&\lesssim&  \left(\sum_{K}\int_{K}\boldsymbol{\tau}_{v}:\boldsymbol{\tau}_{v}d\mathbf{x}\right)^{\frac{1}{2}}
 \left(\sum_{K}\int_{K}\varepsilon(\mathbf{v}):\varepsilon(\mathbf{v})d\mathbf{x}\right)^{\frac{1}{2}}\\
%&\lesssim&(\sum_{K}\int_{K}\boldsymbol{\tau}_{\mathbf{v}}:\boldsymbol{\tau}_{\mathbf{v}}d\mathbf{x})^{\frac{1}{2}}
%  (\sum_{K}\int_{K}\boldsymbol{\tau}_{\mathbf{v}}:\boldsymbol{\tau}_{\mathbf{v}}d\mathbf{x})^{\frac{1}{2}}\\
 &\lesssim&\sum_{K}\int_{K}\boldsymbol{\tau}_v:\boldsymbol{\tau}_v d\mathbf{x}\lesssim
  \int_{\Omega}\boldsymbol{\tau}_v:\varepsilon(\mathbf{v})d\mathbf{x},
   \end{eqnarray*}
where in the first inequality the equivalence of the seminorm $|\varepsilon(\cdot)|_{0}$ and the norm $||\cdot||_{1}$ on the space $V$ is used. Then the  uniform stability inequality $({A2}_{h})$ follows from
 \begin{eqnarray*}
 |\mathbf{v}|_{1}\lesssim  \frac{\int_{\Omega}\boldsymbol{\tau}_v:\varepsilon(\mathbf{v})d\mathbf{x}}{\|\boldsymbol{\tau}_v\|_{0}}\leq\sup_{\boldsymbol{\tau} \in \Sigma_{h}^{PS}}\frac{\int_{\Omega}\boldsymbol{\tau}:\varepsilon(\mathbf{v})d\mathbf{x}}{\|\boldsymbol{\tau}\|_{0}}
\ \,\,\, \mbox{for all } \mathbf{v}\in
 V_{h}.
 \end{eqnarray*}
\end{proof}

 Combining Theorem 4.1 and Theorem 4.2, we immediately
 have the following uniform error estimates.

\begin{theorem}
 Let $(\boldsymbol{\sigma},\mathbf{u})\in \Sigma\times V$ be the solution of the
 variational problem (\ref{weak1.a})(\ref{weak1.b}). Under  the same condition as in Lemma 4.1,  the discretization problem
 (\ref{discreteweak1.a})(\ref{discreteweak1.b}) admits a unique solution 
 $(\boldsymbol{\sigma}_{h},\mathbf{u}_{h})\in \Sigma_{h}^{PS}\times V_{h}$ such that 
 \begin{equation}\label{estimate-PS}
 \|\boldsymbol{\sigma}-\boldsymbol{\sigma}_{h}\|_{0}+|\mathbf{u}-\mathbf{u}_{h}|_{1}
 \lesssim
 \inf_{\boldsymbol{\tau}\in \Sigma_{h}^{PS}}\|\boldsymbol{\sigma}-\boldsymbol{\tau}\|_{0}+\inf_{\mathbf{v}\in
 V_{h}}|\mathbf{u}-\mathbf{v}|_{1}.
 %\lesssim h(||\boldsymbol{\sigma}||_{1}+||\mathbf{u}||_{2}).
 \end{equation}
   In addition, let $p_{h}=-\frac{1}{2}tr\boldsymbol{\sigma}_{h}$ be the approximation of the pressure $p=-(\mu+\lambda)\mbox{div}\mathbf{u}=-\frac{1}{2}tr\boldsymbol{\sigma}$,  then it holds
 \begin{equation}
||p-p_{h}||_{0}\lesssim  \inf_{\boldsymbol{\tau}\in \Sigma_{h}^{PS}}\|\boldsymbol{\sigma}-\boldsymbol{\tau}\|_{0}+\inf_{\mathbf{v}\in
 V_{h}}|\mathbf{u}-\mathbf{v}|_{1}.
\end{equation}   

\end{theorem}
\begin{remark}
Here we recall that ``$\lesssim$'' denotes ``$\leq C$ ''with $C$ a positive constant independent of $\lambda$ and   $h$.
\end{remark}

\begin{remark}
From the standard interpolation theory, the right side terms of (\ref{estimate-PS}) can be further bounded from above by $ Ch(||\boldsymbol{\sigma}||_{1}+||\mathbf{u}||_{2}).$
\end{remark}

\subsection*{4.2. Error analysis  for ECQ4}

Since  the stress mode of   ECQ4  is actually a modified version of PS's with a perturbation term (see Remark 3.4), the stability analysis for ECQ4 can be carried out by following  a  similar routine.  However, due to the coupling of the constant term with  higher order terms,  we need to introduce the mesh condition proposed by Shi  \cite{Shi1984} (Figure 4):

 \noindent {{\bf Condition (A)}} 
 The distance $d_{K}(d_{K}=2\sqrt{a_{12}^{2}+b_{12}^{2}})$ between the midpoints of the diagonals of $K\in T_{h}$
 (Figure 2)
 is of order $o(h_{K})$ uniformly for all elements $K$ as $h\rightarrow
 0$.
 
% \noindent {\bf{Condition (B)}} 
% The distance $d_{K}(d_{K}=2\sqrt{a_{12}^{2}+b_{12}^{2}})$ between the midpoints of the diagonals of $K\in T_{h}$
% is of order $O(h_{K}^{2})$ uniformly for all elements $K$ as $h\rightarrow
% 0$.

%  In
% \cite{ARNOLD.D;BOFFI.D;FALK.R2002,Rannacher.R;Turek.S1992},  a family of quadrilateral meshes is called asymptotically parallelogram
%if it satisfies the Angle Condition $\sigma_{K}= O(h_{K})$, namely if $\sigma_{K}/ h_{K}$ is uniformly bounded for all the elements in all
%the meshes, where $\sigma_{K}:=\max(|\pi-\theta_{1}|,|\pi-\theta_{2}|)$ denotes the deviation of a quadrilateral from a parallelogram with $\theta_{1}$ the angle between the
%outward normals of two opposite sides of $K$ and $\theta_{2}$  the angle between the outward
%normals of the other two sides.  In \cite{Chou.S;He.S2002} the Angle Condition, under the shape regular condition (\ref{CR}) and assuming   $h$ is sufficiently small,  was shown to be equivalent to   Condition (B), or the Bi-Section Condition.  The Angle Condition or Condition (B) 
%ensures that the mesh subdivisions will converge to a set of
%parallelograms, and they will automatically hold when mesh subdivisions are constructed by
%bisections.
%We also refer to \cite{MING.P;SHI.Z2002a} for  equivalence of several well-known shape regular mesh conditions.
%
 
\begin{figure}[!h]
 \vspace{2.5cm}
\begin{center}
\setlength{\unitlength}{1cm}
\begin{picture}(7,3)
\put(2,2.2){\line(6,1){3.8}}     \put(2,2.2){\line(1,5){0.4}}
\put(2.4,4.2){\line(2,1){2}}     \put(4.4,5.2){\line(3,-5){1.4}}
\put(2,2.2){\circle*{0.15}}      \put(2.4,4.2){\circle*{0.15}}
\put(4.4,5.2){\circle*{0.15}}    \put(5.8,2.8){\circle*{0.15}}
\put(1.7,1.9){\bf{$Z_1$}}        \put(5.8,2.5){\bf{$Z_2$}}
\put(4.4,5.4){\bf{$Z_3$}}        \put(2.1,4.4){\bf{$Z_4$}}
\put(2,2.2){\line(4,5){2.4}}     \put(2.4,4.2){\line(5,-2){3.4}}
\put(3.2,3.7){\circle*{0.15}}    \put(4.1,3.52){\circle*{0.15}}
\put(3.2,3.7){\line(5,-1){0.9}}  \put(2.7,3.7){\bf{$O_1$}}
\put(4.2,3.6){\bf{$O_2$}}        \put(3.5,3.3){\bf{$d_k$}}
\end{picture}
\end{center}
\vspace{-2.3cm} \caption{The distance $d_{K}$}
\end{figure}
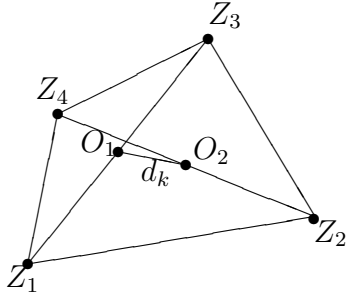

For the uniform discrete kernel-coercivity $({A1}_{h})$ we need the following lemma.

\begin{lemma}
 (\cite{Zhang1997})
  Let the partition $T_{h}$ satisfy (\ref{partition condition}). Then for
 any $q\in W_{h}$ and $\mathbf{v}\in V_{h}$, there exists $\mathbf{v}^{b}\in B_{h}$ such that
 \begin{equation}\label{quote1}
 \int_{\Omega}(q-\Pi_{0}q)\ (\mathbf{div} \mathbf{v}+\tilde{\mathbf{div}}
 \mathbf{v}^{b})d\mathbf{x}=\|q-\Pi_{0}q\|_{0}^{2},
 \end{equation}
 \begin{equation}\label{quote2}
 |\mathbf{v}^{b}|_{1,h}^{2}\lesssim \|q-\Pi_{0}q\|_{0}^{2}+2|\mathbf{v}|_{1}^{2},
 \end{equation}
 where   $\Pi_{0}:L^{2}(\Omega)\rightarrow \bar{W}_{h}$ is defined by
 $\Pi_{0}q|_{K}:=\frac{1}{4}\int_{K}J_{K}^{-1}qd\mathbf{x}$. % the usual $L^{2}$-projection operator. 
\end{lemma}

We immediately have the following result.
\begin{lemma}
  Let the partition $T_{h}$ satisfy (\ref{partition condition}) and  {\bf Condition (A)}. Then
 it holds
\begin{equation}
(1-o(1)) \|q\|_{0}\lesssim \sup _{\mathbf{v}\in V_{h}, \mathbf{v}^{b}\in B_{h}}\frac{\int_{\Omega}q\ (\mathbf{div} \mathbf{v}+\mathbf{div}_{h}
 \mathbf{v}^{b})d\mathbf{x}}{|\mathbf{v}+\mathbf{v}^{b}|_{1,h}}\, \ \ \mbox{for all }
 q \in W_{h},
\end{equation}
where $o(1)$ means $o(1)\rightarrow0$ as $h\rightarrow0$, and $\mathbf{div}_{h}$ denotes piecewise divergence with respect to $T_{h}$.
\end{lemma}
\begin{proof}
 For any $q\in W_{h}$,  we can write  
 $$
q|_{K}\circ F_{K}=q_{0}^{K}+q_{1}^{K}\xi+q_{2}^{K}\eta.
 $$
 Then it is easy to know that
 $
 \Pi_{0}q|_{K}=q_{0}^{K}. 
 $ 
 
 By Lemma 4.1, there exists $\mathbf{v}\in V_{h}$ such that  (\ref{Q1-P0infsup1}) hold with 
 $\bar{q}=\Pi_{0}q$.  On the other hand, from Lemma 4.3 there exists $\mathbf{v}^{b}$ satisfying (\ref{quote1})(\ref{quote2}).
 
 Since it holds the relations
\begin{equation*}
  \int_{\Omega}\Pi_{0}q\ \tilde{\mbox{div}} \mathbf{v}^{b}d\mathbf{x}=0,
\end{equation*}
and
\begin{equation*}
\int_{\Omega} (q-\Pi_{0}q)\ \mbox{div} \mathbf{v}^{b}d\mathbf{x} =\int_{\Omega}(q-\Pi_{0}q)\ \tilde{\mbox{div}} \mathbf{v}^{b}d\mathbf{x},
\end{equation*}
 it follows from (\ref{quote1}), (\ref{quote2}),  (\ref{Q1-P0infsup1}) that
\begin{equation}\label{div-v+vb}
\begin{array}{ll}
 & |\mathbf{v}+\mathbf{v}^{b}|_{1,h}\|q\|_{0}+\int_{\Omega}\Pi_{0}q\ \mbox{div} \mathbf{v}^{b}d\mathbf{x}\\
  \lesssim & ||q-\Pi_{0}q||_{0}^{2}+||\Pi_{0}q||_{0}^{2}+\int_{\Omega}\Pi_{0}q\ \mbox{div} \mathbf{v}^{b}d\mathbf{x}\\
 \lesssim&\int_{\Omega}(q-\Pi_{0}q)\ (\mbox{div} \mathbf{v}+\tilde{\mbox{div}} \mathbf{v}^{b})d\mathbf{x}+\int_{\Omega}\Pi_{0}q\ \mbox{div} \mathbf{v}d\mathbf{x}+\int_{\Omega}\Pi_{0}q\ \mbox{div} \mathbf{v}^{b}d\mathbf{x}\\
=&\int_{\Omega}q\ (\mbox{div} \mathbf{v}+\mbox{div} \mathbf{v}^{b})d\mathbf{x}.
 \end{array}
 \end{equation}
For the second term in the first line of (\ref{div-v+vb}), from (\ref{vB-form}), Remark 3.6, (\ref{partition condition}), {\bf Condition (A)}, and Remark 4.1, we have
\begin{eqnarray*}
 |\int_{\Omega}\Pi_{0}q\ \mbox{div} \mathbf{v}^{b}d\mathbf{x}|
 &=& \sum\limits_{K\in T_{h}}\frac{4}{3}|(b_{12}(u_{\xi}-u_{\eta})+a_{12}(v_{\eta}-v_{\xi})) |\Pi_{0}q|_{K}|\\
 &\lesssim &\sum\limits_{K\in T_{h}}\frac{|b_{12}|+|a_{12}|}{h_{K}}(|u_{\xi}|+|u_{\eta}|+|v_{\eta}|+|v_{\xi}|)||\Pi_{0}q||_{0,K}\\
 &\leq &\sum\limits_{K\in T_{h}} o(1)|\mathbf{v}+\mathbf{v}^{b}|_{1,K}\|\Pi_{0}q\|_{0,K}
 \leq o(1)|\mathbf{v}+\mathbf{v}^{b}|_{1,h}\|q\|_{0},
\end{eqnarray*}
which, together with (\ref{div-v+vb}), yields the desired result.
\end{proof}

%Simarly, replacing {\bf Condition (A)} by Condition (B), we  have the following lemma.
%\begin{lemma}
% Let the partition $T_{h}$ satisfy (\ref{partition condition}) and  Condition (B). Then
% it holds
%\begin{equation}
%(1-O(h)) \|q\|_{0}\lesssim \sup _{\mathbf{v}\in V_{h}, \mathbf{v}^{b}\in B_{h}}\frac{(\mbox{div} \mathbf{v}+\mbox{div}
% \mathbf{v}^{b},q)}{|\mathbf{v}+\mathbf{v}^{b}|_{1,h}}\, \ \  \mbox{for all }
% q \in W_{h}.
%\end{equation}
%\end{lemma}

 From Lemmas 4.4 and 4.5 we know that, under the assumptions in the lemmas, the inf-sup condition 
 \begin{equation}
\|q\|_{0}\lesssim \sup _{\mathbf{v}\in V_{h}, \mathbf{v}^{b}\in B_{h}}\frac{(\mbox{div} \mathbf{v}+\mbox{div}
 \mathbf{v}^{b},q)}{|\mathbf{v}+\mathbf{v}^{b}|_{1,h}} \, \ \ \mbox{for all }
 q \in W_{h}
\end{equation}
holds
  when the mesh size $h$ is small enough.
    
 Therefore, following the same routine as in the proof of  Theorem 4.2, we arrive at the following  result.
\begin{theorem}
 Under {\bf Condition (A)} and the same conditions as in Lemma 4.1,
 the uniform discrete kernel-coercivity condition ($\mathrm{A1}_{h}$)
 holds for ECQ4 with $\Sigma_{h}=\Sigma_{h}^{EC}$ and sufficiently small mesh size $h$.
\end{theorem}

Next we show the discrete inf-sup condition ($\mathrm{A2}_{h}$) holds for the ECQ4 finite element.
Notice that {\bf Condition (A)} states
 \begin{equation}
\max\{|a_{12}|,|b_{12}|\}=o(h_{K}), \ \ \max\{|J_{1}|,|J_{2}|\}=o(h_{K}^{2}).
\end{equation}
Recall  the element geometric properties (\ref{a1b2-hK})-(\ref{JK-hK}), namely 
\begin{equation}\label{element geometry}
a_{1}\approx b_{2}\approx h_{K},  \ \max\{a_{2},b_{1}\}\lesssim O(h_{K}), \ J_{0}\approx h_{K}^{2}.
\end{equation}
 This allows us to view   all the terms involving one of the factors $a_{12},b_{12},J_{1},J_{2}$   as higher-order terms. In this sense, the ECQ4 stress mode (\ref{ECQ4stress1}) is actually a higher-order oscillation of the PS stress mode (\ref{PSstress1}) (cf. Remark \ref{stress-osc}). Thus, under {\bf Condition (A)} Lemmas \ref{PS-stress-ineq}-\ref{tau-epsilon-relation} also hold for ECQ4 stress space $\Sigma_{h}^{EC}$. 
 
 As a result, 
  we have the following stability result for the ECQ4 finite element.
\begin{theorem}
 Let the partition $T_{h}$ satisfy the shape-regularity condition (\ref{partition condition}) and {\bf Condition (A)}. Then
 the uniform discrete inf-sup condition (${A2}_{h}$) holds with 
 $\Sigma_{h}=\Sigma_{h}^{EC}.$
\end{theorem}

Combining Theorem 4.4 and Theorem 4.5, we immediately
 have the following uniform error estimates for the ECQ4 finite element:

\begin{theorem}
 Let $(\boldsymbol{\sigma},\mathbf{u})\in \Sigma\times V$ be the solution of the
 variational problem (\ref{weak1.a})(\ref{weak1.b}). Under  the same conditions as in Theorem 4.4,  the discretization problem
 (\ref{discreteweak1.a})(\ref{discreteweak1.b}) admits a unique solution 
 $(\boldsymbol{\sigma}_{h},\mathbf{u}_{h})\in \Sigma_{h}^{EC}\times V_{h}$ such that 
 \begin{equation}\label{estimate-EC}
 \|\boldsymbol{\sigma}-\boldsymbol{\sigma}_{h}\|_{0}+|\mathbf{u}-\mathbf{u}_{h}|_{1}
 \lesssim
 \inf_{\boldsymbol{\tau}\in \Sigma_{h}^{EC}}\|\boldsymbol{\sigma}-\boldsymbol{\tau}\|_{0}+\inf_{\mathbf{v}\in
 V_{h}}|\mathbf{u}-\mathbf{v}|_{1}.
   \end{equation}
   In addition, let $p_{h}=-\frac{1}{2}tr\boldsymbol{\sigma}_{h}$ be the approximation of the pressure $p=-(\mu+\lambda)\mbox{div}\mathbf{u}=-\frac{1}{2}tr\boldsymbol{\sigma}$,  then it holds
 \begin{equation}
||p-p_{h}||_{0}\lesssim  \inf_{\boldsymbol{\tau}\in \Sigma_{h}^{EC}}\|\boldsymbol{\sigma}-\boldsymbol{\tau}\|_{0}+\inf_{\mathbf{v}\in
 V_{h}}|\mathbf{u}-\mathbf{v}|_{1}.
\end{equation}   
\end{theorem}
%\begin{remark}
%Here we recall that ``$\lesssim$'' denotes ``$\leq C$ ''with $C$ a positive constant which is bounded as $\lambda\rightarrow \infty$ and is independent of $h$.
%\end{remark}
%
%\begin{remark}
%From the standard interpolation theory, the right side terms of (\ref{estimate-EC}) can be further bounded from above by $ Ch(||\boldsymbol{\sigma}||_{1}+||\mathbf{u}||_{2}).$
%\end{remark}

 \setcounter{remark}{0}\setcounter{lemma}{0} \setcounter{theorem}{0}
 \setcounter{section}{5} \setcounter{equation}{0}
\section*{5. \ Equivalent EAS schemes}
%\subsection{Equivalent higher order hybrid stress schemes}

%\subsection{Equivalent EAS schemes}

 By following the basic idea of \cite{Piltner-Taylor1995,Piltner-Taylor1999,Piltner2000}, this part is devoted to the equivalence between the hybrid stress finite element method and some 
 enhanced strains finite element scheme.

 The equivalent enhanced strains method is based on   the following
 modified Hu-Washizu functional:
 \begin{eqnarray*}
  \Pi(\boldsymbol{\tau},\mathbf{v},\boldsymbol{\gamma},\boldsymbol{\gamma}^{b})
 &=&-\frac{1}{2}b(\boldsymbol{\gamma},\boldsymbol{\gamma})
 +\sum_{K}\{\int_{K}\boldsymbol{\tau}:(\boldsymbol{\gamma}-\varepsilon(\mathbf{v})-\boldsymbol{\gamma}^{b})d\mathbf{x}\\
 &&-\oint_{\boldsymbol{\gamma}_{N}\cap\partial K}\mathbf{g}\cdot \mathbf{v}ds
 -\int_{K}f\cdot \mathbf{v}d\mathbf{x}\},
 \end{eqnarray*}
 where
\begin{eqnarray*}
  b(\boldsymbol{\alpha},\boldsymbol{\beta})=\int_{\Omega}\boldsymbol{\alpha}:\mathbb{C}\boldsymbol{\beta}d\mathbf{x}=\int_{\Omega}(2\mu\boldsymbol{\alpha}:\boldsymbol{\beta}+\lambda tr\boldsymbol{\alpha}
  tr\boldsymbol{\beta})d\mathbf{x},
 \end{eqnarray*}
 $\mathbf{v}\in V_{h}$ is the compatible  displacements given in (\ref{displacement}),
 $\varepsilon(\mathbf{v})=(\nabla \mathbf{v}+\nabla^{T} \mathbf{v})/2$
 denotes the strain caused by the displacement vector $\mathbf{v}$,
 $\boldsymbol{\tau}\in \tilde{\Sigma}_{h}$ is the unconstraint stress tensor with $$
 \tilde{\Sigma}_{h}:=\{\boldsymbol{\gamma}\in  \mathbf{L}^{2}(\Omega;\mathbb{R}_{sym}^{2\times2}): \hat{\boldsymbol{\gamma}}_{ij}=\boldsymbol{\gamma}_{ij}|_{K}\circ F_{K}\in span\{1,\xi,\eta\}\ \ 
 \mbox{for }i,j=1,2, \ K\in T_{h}\},$$ 
$\boldsymbol{\gamma}\in
  \tilde{\Sigma}_{h}$ and $\boldsymbol{\gamma}^{b}\in U_{h}^{b}$ are the independent strain and enhanced strain tensors
 respectively with 
 $$U_{h}^{b}=U_{PS}^{b}:=\{\tilde{\varepsilon}(\mathbf{v}^{b}):\ \mathbf{v}^{b}\in
 B_{h}\}$$
  for the PS finite element, and 
 $$U_{h}^{b}=U_{EC}^{b}:=\{\varepsilon(\mathbf{v}^{b}):\ \mathbf{v}^{b}\in
 B_{h}\}$$
  for the ECQ4 finite element. %\cite{Piltner2000}.

 The variational formulations of the above enhanced strains
 method read as:
 Find $(\boldsymbol{\sigma}_{h},\mathbf{u}_{h},\boldsymbol{\varepsilon}_{h},\boldsymbol{\varepsilon}_{h}^{b})\in  \tilde{\Sigma}_{h}\times V_{h}\times  \tilde{\Sigma}_{h}\times
 U_{h}^{b}$ such that
 \begin{eqnarray}\label{equivalence1}
 \sum_{K}\{\int_{K}\boldsymbol{\tau}:(\boldsymbol{\varepsilon}_{h}-\varepsilon(\mathbf{u}_{h})-\boldsymbol{\varepsilon}_{h}^{b})d\mathbf{x}\}=0\hskip2cm \mbox{for all } \boldsymbol{\tau}\in  \tilde{\Sigma}_{h},
 \end{eqnarray}
 \begin{eqnarray}\label{equivalence2}
 b(\boldsymbol{\gamma},\boldsymbol{\varepsilon}_{h})-\int_{\Omega}\boldsymbol{\gamma}:\boldsymbol{\sigma}_{h}d\mathbf{x}=0\hskip4cm \mbox{for all } \boldsymbol{\gamma}\in  \tilde{\Sigma}_{h},
 \end{eqnarray}
 \begin{eqnarray}\label{equivalence3}
\int_{\Omega}\boldsymbol{\sigma}_{h}:\varepsilon(\mathbf{v})d\mathbf{x}
 =\sum_{K}\{\int_{K} f\cdot \mathbf{v}d\mathbf{x}+\oint_{\Gamma_{N}\cap\partial K}\mathbf{g}\cdot \mathbf{v}ds\}\  \mbox{for all } \mathbf{v}\in V_{h},
 \end{eqnarray}
 \begin{eqnarray}\label{equivalence4}
\int_{\Omega}\boldsymbol{\sigma}_{h}:\boldsymbol{\gamma}^{b}d\mathbf{x}=0\hskip6cm  \mbox{for all } \boldsymbol{\gamma}^{b}\in U_{h}^{b}.
 \end{eqnarray}

 We claim that the hybrid stress finite element scheme (\ref{discreteweak1.a})(\ref{discreteweak1.b})  for PS and ECQ4 is equivalent to the scheme (\ref{equivalence1})-(\ref{equivalence4}) in the sense that the stress and displacement solution,  $(\boldsymbol{\sigma}_{h},\mathbf{u}_{h})$,  of the latter enhanced strains scheme,  also satisfy the equations (\ref{discreteweak1.a})(\ref{discreteweak1.b}).

 In fact, we  decompose
 $ \tilde{\Sigma}_{h}$ as $\tilde{\Sigma}_{h}=\Sigma_{h}\oplus ( \tilde{\Sigma}_{h} \diagdown \Sigma_{h})$, where $\Sigma_{h}=\Sigma_{h}^{PS}$
 for the PS finite element and   $\Sigma_{h}=\Sigma_{h}^{EC}$ for ECQ4. It is easy to see that the relation (\ref{equivalence4})  indicates $\boldsymbol{\sigma}_{h} \in \Sigma_{h}$.  Thus (\ref{equivalence4}) is just the same as (\ref{discreteweak1.b}). 

On the other hand, by using the decomposition of $ \tilde{\Sigma}_{h}$,  the equation
 (\ref{equivalence1}) leads to:
\begin{eqnarray}\label{equivalence1.a}
 \sum_{K}\{\int_{K}\boldsymbol{\tau}:(\boldsymbol{\varepsilon}_{h}-\varepsilon(\mathbf{u}_{h})-\boldsymbol{\varepsilon}_{h}^{b})d\mathbf{x}\}=0\hskip1cm  \mbox{for all } \boldsymbol{\tau}\in  \tilde{\Sigma}_{h}\diagdown\Sigma_{h},
 \end{eqnarray}
\begin{eqnarray}\label{equivalence1.b}
 \sum_{K}\{\int_{K}\boldsymbol{\tau}:(\boldsymbol{\varepsilon}_{h}-\varepsilon(\mathbf{u}_{h})-\boldsymbol{\varepsilon}_{h}^{b})d\mathbf{x}\}=0\ \mbox{for all } \boldsymbol{\tau}\in \Sigma_{h}.
 \end{eqnarray}
 Since $2\mu\boldsymbol{\varepsilon}_{h}+\lambda tr\boldsymbol{\varepsilon}_{h} \mathbf{I}-\boldsymbol{\sigma}_{h} \in  \tilde{\Sigma}_{h}$, from (\ref{equivalence2}) we get
 $\boldsymbol{\sigma}_{h}=2\mu\boldsymbol{\varepsilon}_{h}+\lambda tr\boldsymbol{\varepsilon}_{h} \mathbf{I}$ or
 $\boldsymbol{\varepsilon}_{h}= \frac{1}{2\mu}[\boldsymbol{\sigma}_{h}-\frac{\lambda}{2(\mu+\lambda)}tr\boldsymbol{\sigma}_{h} \mathbf{I}]$.
 Substitute this into (\ref{equivalence1.b}), we then get an
 equation as same as (\ref{discreteweak1.a}).  Hence, the equivalence follows.  
 
Notice that one can solve  $\varepsilon_{h}^{b}$ from the equation (\ref{equivalence1.a}).

 \begin{remark}
  As shown in \cite{Piltner2000,Xie-Zhou2004,Xie-Zhou2008}, we also have two higher-order hybrid stress finite element schemes equivalent to the schemes of PS and ECQ4, respectively.  More precisely, the higher-order  schemes are given as:    Find $(\tilde{\boldsymbol{\sigma}}_{h},\tilde{\mathbf{u}}_{h},\mathbf{u}_{h}^{b})\in \tilde{\Sigma}_{h}\times V_{h}\times B_{h}$  such that
 \begin{equation*}%\label{discreteweak1.a new}
 a(\tilde{\boldsymbol{\sigma}}_{h},\boldsymbol{\tau})-\int_{\Omega}\boldsymbol{\tau}:\left(\varepsilon(\tilde{\mathbf{u}}_{h})+\varepsilon_{M}(\mathbf{u}_{h}^{b})\right) d\mathbf{x}=0\ \  \,\mbox{for all }\boldsymbol{\tau}\in
\tilde{\Sigma}_{h},
\end{equation*}
\begin{equation*}%\label{discreteweak1.b new}
 \int_{\Omega}\tilde{\boldsymbol{\sigma}}_{h}:\left(\varepsilon(\mathbf{v})+\varepsilon_{M}(\mathbf{v}^{b})\right) d\mathbf{x}
 =F(\mathbf{v})\ \ \, 
 \mbox{for all } \mathbf{v}\in V_{h},\ \mathbf{v}^{b}\in B_{h},
 \end{equation*}
where $\varepsilon_{M}=\tilde{\varepsilon}$ for the PS case and $\varepsilon_{M}=\varepsilon$ for the ECQ4 case. The equivalence is in the sense that the solutions of the scheme (\ref{discreteweak1.a})-(\ref{discreteweak1.b}) for PS and ECQ4 and of the above higher-order scheme satisfy
$$ \tilde{\boldsymbol{\sigma}}_{h}=\boldsymbol{\sigma}_{h}\quad\mbox{and}\quad \tilde{\mathbf{u}}_{h}=\mathbf{u}_{h}.$$
 In fact, due to the constraints (\ref{PSconstraint})-(\ref{ECconstraint}), we can view the higher-order scheme  as an unconstrained one derived from  the constrained scheme (\ref{discreteweak1.a})-(\ref{discreteweak1.b}), with $\mathbf{u}_{h}^{b}\in   B_{h}$ being a Lagrange multiplier.
 \end{remark}

 \begin{remark}
 Notice that in the hybrid stress finite element scheme (\ref{discreteweak1.a})-(\ref{discreteweak1.b}), a term like $\mathbb{C}^{-1} $ is involved.  Thus for non-linear
  problems where  $\mathbb{C}$ is not  a constant modulus tensor, it is not convenient to implement the hybrid finite element method, while for  the    the enhanced strains
  method, this is not a difficulty, since one does not need to compute $\mathbb{C}^{-1} $.   However, owing to the equivalence shown above,     the hybrid finite element technology with PS and ECQ4 is easily extended to non-linear
  problems.
 \end{remark}

 \setcounter{remark}{0}\setcounter{lemma}{0} \setcounter{theorem}{0}
 \setcounter{section}{6} \setcounter{equation}{0}
\section*{6. \ Uniform a posteriori error estimates for hybrid methods}
\subsection*{6.1. \ A posteriori error analysis}
By following the same routine as in \cite{Braess-Carstensen-Reddy2004, Carstensen2005, Carstensen-Hu-Orlando2007},  one derives  the  computable upper bound 
\begin{equation}\label{estimator}
\eta_{h}^{2}:=\sum\limits_{K\in T_{h}}\|h_{K}(\mathbf{f}+{\bf div}
\boldsymbol{\sigma}_{h})\|_{0,K}^{2}+
 \|\mathbb{C}^{-1}\boldsymbol{\sigma}_{h}-\varepsilon(\mathbf{u}_{h})\|_{0,\Omega}^{2}+
 \sum\limits_{E\in \mathcal{E}_{0}\bigcup \mathcal{E}_{N}}h_{E}\|[\boldsymbol{\sigma}_{h} \mathbf{n}_{E}]\|_{0,E}^{2}
\end{equation}
for the  error $ \|\boldsymbol{\sigma}-\boldsymbol{\sigma}_{h}\|_{0}^{2}+|\mathbf{u}-\mathbf{u}_{h}|_{1}^{2} $ of the hybrid finite element methods. 
Here $\mathcal{E}_{0}$ denotes the set of all interior edges of $T_{h}$, $\mathcal{E}_{N}$ the set of  all edges on the boundary $\Gamma_{N}$,   $h_{E}$  the length of an edge $E\in \mathcal{E}:=\mathcal{E}_{0}\bigcup\mathcal{E}_{N}$,  $\mathbf{n}_{E}$ the unit  normal along $E$,  and $[\boldsymbol{\sigma}_{h}\mathbf{n} _{E}]$   the jump
of $\boldsymbol{\sigma}_{h}\mathbf{n}$ on $E$, especially for  $E\in \mathcal{E}_{N}$,  $[\boldsymbol{\sigma}_{h}\mathbf{n} _{E}]:=\boldsymbol{\sigma}_{h}\mathbf{n} _{E}-\mathbf{g}$. 

 We first define an operator $\mathbb{A}:\ \Sigma\times V\ \rightarrow (\Sigma\times V)' $ by
\begin{equation*}
<\mathbb{A}(\boldsymbol{\sigma},\mathbf{u}),(\boldsymbol{\tau},\mathbf{v})>:=a(\boldsymbol{\sigma},\boldsymbol{\tau})-\int_{\Omega}\boldsymbol{\sigma}:\varepsilon(\mathbf{v})d\mathbf{x}-\int_{\Omega}\boldsymbol{\tau}:\varepsilon(\mathbf{u})d\mathbf{x}
\end{equation*}
for all $\boldsymbol{\sigma},\boldsymbol{\tau}\in \Sigma $ and $\mathbf{u}, \mathbf{v}\in V $.  Then, from (A1), (A2) and Theorem 2.2 we immediately get
\begin{lemma}
The operator $\mathbb{A}$ defined as above is bounded and bijective, and the operator norms of  $\mathbb{A}$ and  $\mathbb{A}^{-1}$ are independent of  $\lambda$ and $h$.% in a sense that
%\begin{equation}

%\end{equation}
\end{lemma}

We  need the following weak interpolation operator \cite{Bernardi-Girault1998}.
\begin{lemma}
 Let the partition $T_{h}$ satisfy (\ref{partition condition}) . Then there exists an operator $\mathcal{J}:\ V\rightarrow V_{h} $ such that, for all $\mathbf{v}\in V$,
\begin{equation}\label{Clement}
||h_{\mathcal{T}}^{-1}(\mathbf{v}-\mathcal{J}\mathbf{v})||_{0}+||h_{\mathcal{E}}^{-1/2}(\mathbf{v}-\mathcal{J}\mathbf{v})||_{0,\mathcal{E}}
\lesssim |\mathbf{v}|_{1}.
\end{equation}
\end{lemma}

In light of this lemma, we have the following a posteriori error estimate for the hybrid finite element scheme (\ref{discreteweak1.a})-(\ref{discreteweak1.b}).

\begin{theorem}
 Let the partition $T_{h}$ satisfy (\ref{partition condition}) .  Then it holds 
 \begin{equation}\label{posteriori}
 \|\boldsymbol{\sigma}-\boldsymbol{\sigma}_{h}\|_{0}+|\mathbf{u}-\mathbf{u}_{h}|_{1}
 \lesssim
\eta_{h}.
  \end{equation}
\end{theorem}
\begin{remark}
Here we recall that ``$\lesssim$'' denotes ``$\leq C$ ''with $C$ a positive constant which is bounded as $\lambda\rightarrow \infty$ and is independent of $h$.

\end{remark}

\begin{remark}
In fact, the reliable error estimate in Theorem 6.1 is efficient as well in a sense that the estimate
\begin{equation}
\sum\limits_{K\in T_{h}}\|h_{K}(\mathbf{f}+{\bf div}
\boldsymbol{\sigma}_{h})\|_{0,K}^{2}+
 \sum\limits_{E\in \mathcal{E}}h_{E}\|[\boldsymbol{\sigma}_{h} \mathbf{n}_{E}]\|_{0,E}^{2}\lesssim
\|\boldsymbol{\sigma}-\boldsymbol{\sigma}_{h}\|_{0}^{2}+ |\mathbf{u}-\mathbf{u}_{h}|_{1}^{2}+osc(\mathbf{f},T_{h})^{2}
\end{equation}
holds, where  $osc(\mathbf{f},T_{h})^{2}:=\sum\limits_{K\in T_{h}}\|h_{K}(\mathbf{f}-\mathbf{f}_{h})\|_{0,K}^{2}$ for the $T_{h} $ piecewise constant integral means $\mathbf{f}_{h}$.  This can be obtained by
following  similar arguments in \cite{Verfurth1996}.
\end{remark}

\noindent {\it Proof of Theorem 6.1.} The desired result can be obtained by following the same routine as in   in \cite{Braess-Carstensen-Reddy2004}. Here for completeness we   give a proof.

In fact, the stability of $\mathbb{A}$ in Lemma 6.1 ensures that 
\begin{eqnarray*}
\|\boldsymbol{\sigma}-\boldsymbol{\sigma}_{h}\|_{0}+|\mathbf{u}-\mathbf{u}_{h}|_{1}
\lesssim\sup\limits_{\boldsymbol{\tau}\in\Sigma,\mathbf{v}\in V}
\frac{<\mathbb{A}(\boldsymbol{\sigma}-\boldsymbol{\sigma}_{h},\mathbf{u}-\mathbf{u}_{h}),(\boldsymbol{\tau},\mathbf{v})>}{||\boldsymbol{\tau}||_{0}+|\mathbf{v}|_{1}}.
\end{eqnarray*}
With the relation $\boldsymbol{\sigma}=\mathbb{C}^{-1}\varepsilon(\mathbf{u})$ and the Galerkin orthogonality $\int_{\Omega}(\boldsymbol{\sigma}-\boldsymbol{\sigma}_{h}):\varepsilon(\mathcal{J}\mathbf{v})d\mathbf{x} =0, $ this equals
$$
\sup\limits_{\boldsymbol{\tau}\in\Sigma,\mathbf{v}\in V}
\frac{\int_{\Omega}(\mathbb{C}^{-1}(\boldsymbol{\sigma}-\boldsymbol{\sigma}_{h})-\varepsilon(\mathbf{u}-\mathbf{u}_{h})):\boldsymbol{\tau}d\mathbf{x}-\int_{\Omega}(\boldsymbol{\sigma}-\boldsymbol{\sigma}_{h}):\varepsilon(\mathbf{v})d\mathbf{x}}{||\boldsymbol{\tau}||_{0}+|\mathbf{v}|_{1}}$$
$$=\sup\limits_{\boldsymbol{\tau}\in\Sigma,\mathbf{v}\in V}
\frac{\int_{\Omega}(\varepsilon(\mathbf{u}_{h})-\mathbb{C}^{-1}\boldsymbol{\sigma}_{h}):\boldsymbol{\tau}d\mathbf{x}-\int_{\Omega}(\boldsymbol{\sigma}-\boldsymbol{\sigma}_{h}):\varepsilon(\mathbf{v}-\mathcal{J}\mathbf{v})d\mathbf{x}}{||\boldsymbol{\tau}||_{0}+|\mathbf{v}|_{1}}.
$$
With Cauchy's inequality and integration by parts, plus Lemma 6.2, this is bounded from above by
$$
{\small
 \sup\limits_{\boldsymbol{\tau}\in\Sigma,\mathbf{v}\in V}
\left(-\sum\limits_{K\in T_{h}}\int_{K}(\mathbf{f}+{\bf div}\boldsymbol{\sigma}_{h}):(\mathbf{v}-\mathcal{J}\mathbf{v})d\mathbf{x}+\sum\limits_{E\in \mathcal{E}}[\boldsymbol{\sigma}_{h}{\bf n}_{E}]\cdot (\mathbf{v}-\mathcal{J}\mathbf{v})ds\right)/|\mathbf{v}|_{1} +
}$$
$$\hskip6cm+||\mathbb{C}^{-1}\boldsymbol{\sigma}_{h}-\varepsilon(\mathbf{u}_{h})||_{0}\hskip1cm \lesssim\eta_{h}. \hskip2cm\Box$$

\subsection*{6.2. \ Numerical verification}

We compute two examples, Examples 2 and 3  in Section 3.3, to verify the reliability and  efficiency of the a posteriori estimator $\eta_{h}$ defined in (\ref{estimator}). We list the results of the relative error $e_{r}$, the relative a posteriori error $\eta_{r}$, and the ratio $\eta_{r}/e_{r}$  in Tables 10-12 and Figure 5 with
$$e_{r}:=\frac{\left(\|\boldsymbol{\sigma}-\boldsymbol{\sigma}_{h}\|_{0}^{2}+|\mathbf{u}-\mathbf{u}_{h}|_{1}^{2}\right)^{1/2}}{\left(\|\boldsymbol{\sigma}\|_{0}^{2}+|\mathbf{u}|_{1}^{2}\right)^{1/2}},\ \ \eta_{r}:=\frac{\eta_{h}}{\left(\|\boldsymbol{\sigma}\|_{0}^{2}+|\mathbf{u}|_{1}^{2}\right)^{1/2}}.$$
The numerical results show that the a posteriori estimator $\eta_{h}$ is reliable and efficient with the ratio $\eta_{r}/e_{r}$ being close to 1 in Example 2 and being around 4 in Example 3. It should be pointed out that in  Figure 5 the mesh-axis coordinates $2, 4, 8,16$ denote  the respective meshes $10\times 2, 20\times4, 40\times8, 80\times16$. 

\begin{table}[!h]\renewcommand{\baselinestretch}{1.25}\small
 \centering
 \caption{\small Numerical results of the a posteriori error estimator   for PS in Example 2}
\begin{tabular}{ccccccccccc}
 \hline
 &&regular&mesh &of  &Figure 3 && irregular&mesh&of &Figure 3\\
 \cline{3-6}\cline{8-11}
 $\nu$ & &$10\times2$ &$20\times4$ &$40\times8$ &$80\times16$ & &$10\times2$ &$20\times4$ &$40\times8$ &$80\times16$ \\\hline
       &$\eta_{r}$(e-4)   &4.3306   &2.1653   &1.0826  &0.5413  & &500.41    &99.415    &21.676   &4.6915\\
 0.49  &$e_{r}$(e-4)   &3.5126   &1.7563   &0.8781  &0.4391  & &452.18    &93.579    &21.203   &5.1370\\
       &$\eta_{r}/e_{r}$        &1.23    &1.23    &1.23  &1.23  & &1.11  &1.06     &1.02    &0.91 \\\hline
       &$\eta_{r}$(e-4)  &4.3300    &2.1650   &1.0825  &0.5413  & &496.40    &98.740    &21.586   &4.6974\\
 0.499 &$e_{r}$(e-4)  &3.5331    &1.7665   &0.8833  &0.4416  & &447.56    &92.648    &20.981   &5.0817\\
       &$\eta_{r}/e_{r}$        &1.23    &1.23   &1.23  &1.23  & &1.11    &1.07    &1.03   &0.92 \\\hline
       &$\eta_{r}$(e-4)  &4.3300    &2.1650   &1.0825  &0.5413  & &496.00    &98.677    &21.585   &4.7100\\
 0.4999&$e_{r}$(e-4)  &3.5352    &1.7676   &0.8838  &0.4419  & &447.10    &92.555    &20.959   &5.0764\\
       &$\eta_{r}/e_{r}$       &1.22    &1.22   &1.22   &1.22   & &1.11    &1.07    &1.03   &0.93\\\hline
       &$\eta_{r}$(e-4)  &4.3300    &2.1650   &1.0825  &0.5413  & &495.96    &98.671    &21.585   &4.7117\\
 0.49999&$e_{r}$(e-4) &3.5354    &1.7677   &0.8839  &0.4419  & &447.05    &92.546    &20.957   &5.0759\\
        &$\eta_{r}/e_{r}$     &1.22     &1.22   &1.22   &1.22  & &1.11    &1.07   &1.03   &0.93\\\hline
\end{tabular}
\end{table}

\begin{table}[!h]\renewcommand{\baselinestretch}{1.25}\small
 \centering
 \caption{\small Numerical results of the a posteriori error estimator   for ECQ4 in Example 2}
\begin{tabular}{ccccccccccc}
 \hline
&&regular&mesh &of  &Figure 3 && irregular&mesh&of &Figure 3\\
\cline{3-6}\cline{8-11}
 $\nu$ & &$10\times2$ &$20\times4$ &$40\times8$ &$80\times16$ & &$10\times2$ &$20\times4$ &$40\times8$ &$80\times16$ \\\hline
       &$\eta_{r}$(e-4)   &4.3306   &2.1653   &1.0826  &0.5413  & &480.69    &86.785    &18.365    &4.0010\\
 0.49  &$e_{r}$(e-4)   &3.5126   &1.7563   &0.8781  &0.4391  & &359.44    &75.927    &17.426    &4.2483\\
       &$\eta_{r}/e_{r}$        &1.23  &1.23   &1.23  &1.23  & &1.34   &1.14    &1.05    &0.94 \\\hline
       &$\eta_{r}$(e-4)  &4.3300    &2.1650   &1.0825  &0.5413  & &480.66    &86.998    &18.514    &4.0744\\
 0.499 &$e_{r}$(e-4)  &3.5331    &1.7665   &0.8833  &0.4416  & &359.37    &75.971    &17.436    &4.2495\\
       &$\eta_{r}/e_{r}$       &1.23    &1.23   &1.23  &1.23  & &1.34    &1.15    &1.06    &0.96\\\hline
       &$\eta_{r}$(e-4)  &4.3300    &2.1650   &1.0825  &0.5413  & &480.66    &87.025    &18.538    &4.0941\\
 0.4999&$e_{r}$(e-4)  &3.5352    &1.7676   &0.8838  &0.4419  & &359.37    &75.977    &17.437    &4.2500\\
       &$\eta_{r}/e_{r}$      &1.22   &1.22   &1.22  &1.22  & &1.34    &1.15    &1.06    &0.96\\\hline
       &$\eta_{r}$(e-4)  &4.3300    &2.1650   &1.0825  &0.5413  & &480.66    &87.027    &18.540    &4.0965\\
 0.49999&$e_{r}$(e-4) &3.5354    &1.7677   &0.8839  &0.4419  & &359.37    &75.977    &17.437    &4.2501\\
        &$\eta_{r}/e_{r}$     &1.22    &1.22   &1.22  &1.22  & &1.34    &1.15    &1.06    &0.97\\\hline
\end{tabular}
\end{table}

\begin{table}[!h]\renewcommand{\baselinestretch}{1.25}\small
 \centering
 \caption{\small Numerical results of the a posteriori error estimator   in Example 3}
\begin{tabular}{ccccccccccc}
 \hline
 &&regular&mesh &of  &Figure 3 && irregular&mesh&of &Figure 3\\\cline{3-6}\cline{8-11}
 method & &$10\times2$ &$20\times4$ &$40\times8$ &$80\times16$ & &$10\times2$ &$20\times4$ &$40\times8$ &$80\times16$ \\\hline
        &$\eta_{r}$  &0.4260   &0.2152    &0.1081     &0.05420  & &0.6232     &0.3137     &0.1579     &0.0793\\
 PS     &$e_{r}$  &0.1022   &0.0512   &0.0256    &0.0128  & &0.1806     &0.0859     &0.0424    &0.0211\\
        &$\eta_{r}/e_{r}$   &4.17    &4.20     &4.22      &4.23   & &3.45      &3.65       &3.72      &3.75\\\hline
        &$\eta_{r}$ &0.4260   &0.2152    &0.1081     &0.0542  & &0.5938     &0.3154      &0.1610     &0.0812\\
 ECQ4   &$e_{r}$ &0.1022   &0.0512   &0.0256    &0.0128  & &0.1850     &0.0910    &0.0453   &0.0226\\
        &$\eta_{r}/e_{r}$   &4.17    &4.20     &4.22      &4.23   & &3.21      &3.47       &3.55      &3.59\\\hline
\end{tabular}
\end{table}

\begin{figure}[!h]
 \subfigure[Example 2:   $\nu=0.49$]{\includegraphics[width=8cm]{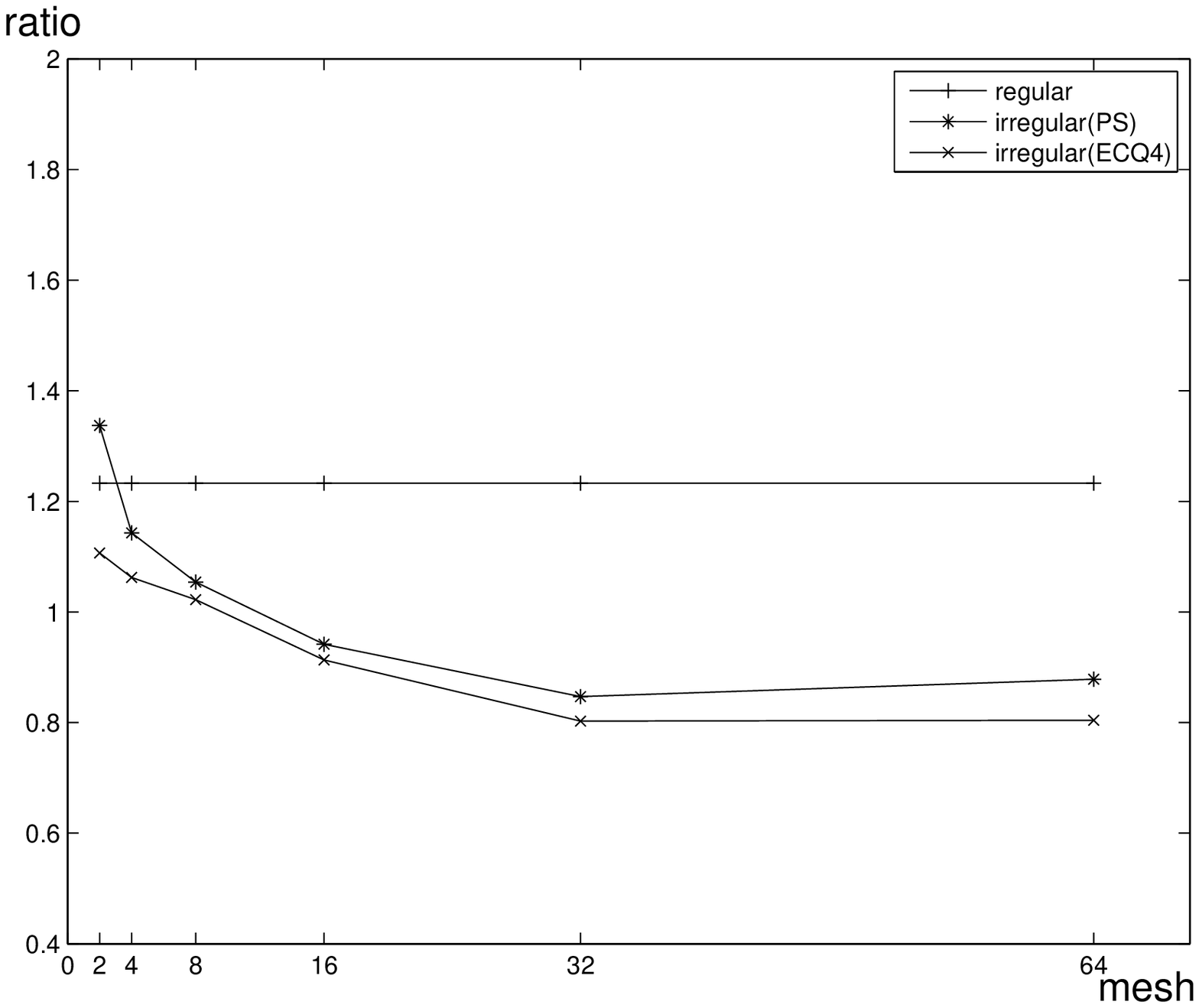}}
 \subfigure[Examle 3]{\includegraphics[width=8cm]{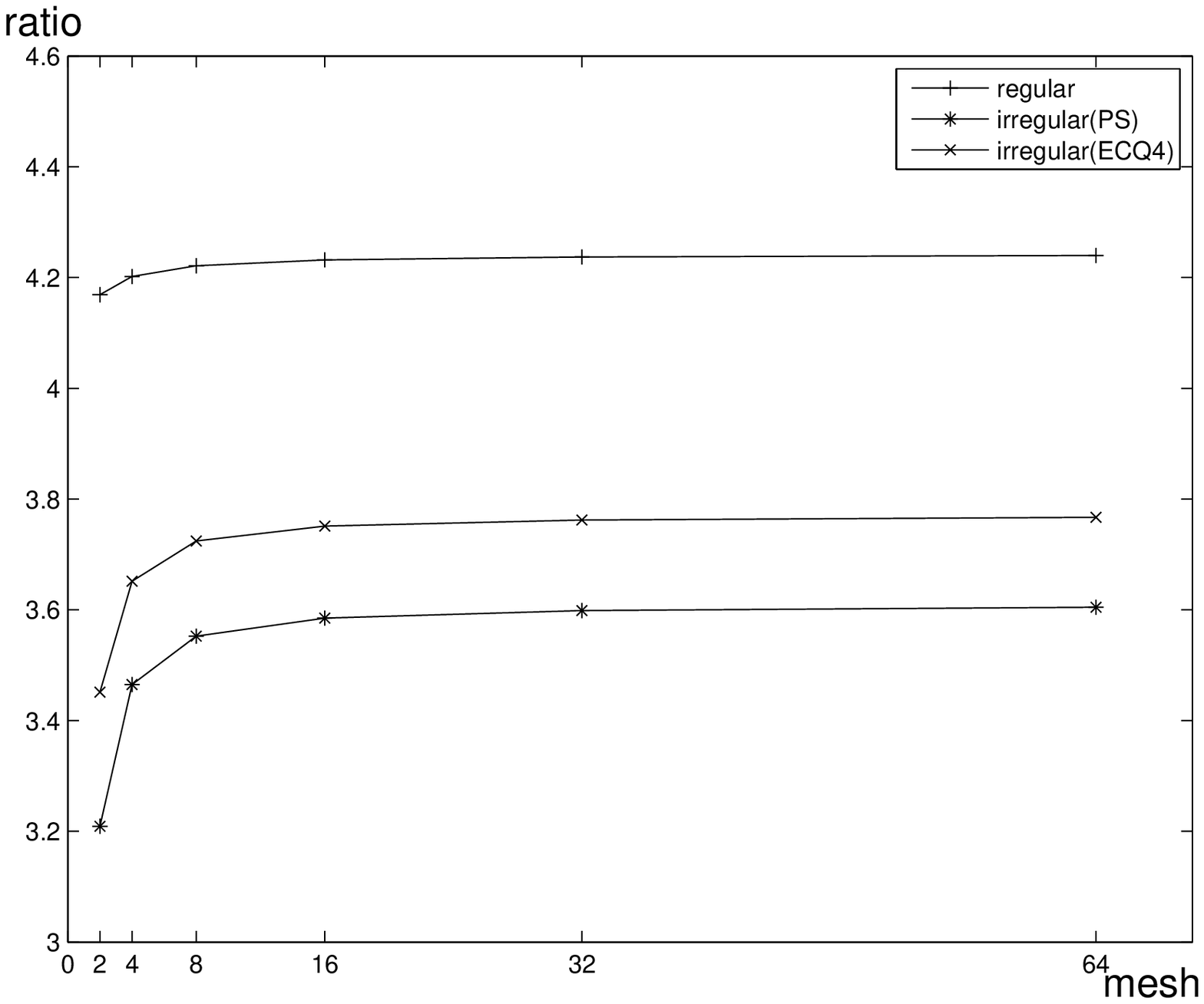}}
 \caption{The ratio $\eta_{r}/e_{r}$ for PS and ECQ4 }
\end{figure}

\section*{Acknowledgements}

This work was  supported  by DFG Research Center MATHEON.
The second author  would like to thank the Alexander von Humboldt Foundation
for the support through the Alexander von Humboldt Fellowship during his stay at
 Department of Mathematics of Humboldt-Universit$\ddot{\mbox{a}}$t zu Berlin, Germany. Part of his work was  supported by the
National Natural Science Foundation of China
   (10771150), the National Basic Research Program of China (2005CB321701),
    and the Program for New Century Excellent Talents in University (NCET-07-0584).  The work of the third author was also partly supported by the WCU program through KOSEF (R31-2008-000-10049-0). 
    
\def\refname

\end{document}